\documentclass[reqno,12pt]{article}
\usepackage{amsmath}
\usepackage{amsthm}
\usepackage{amssymb}
\usepackage{authblk}
\usepackage{enumitem}
\usepackage{todonotes}
\usepackage{hyperref}
\usepackage{tabularx}
\usepackage{algorithm2e}
\usepackage{graphicx,tikz}
\usetikzlibrary{calc,math}
\usepackage{ifthen}
\usepackage{grafy}
\usepackage{listings}
\usepackage[polish,main=english]{babel}
\usepackage{multirow,boldline,makecell}  
\makeatletter
\renewcommand{\p@enumii}{} 
\renewcommand{\p@enumiii}{}
\makeatother

\usepackage[section]{placeins}
\newcommand{\subs}{\subseteq}
\newcommand{\setm}{\setminus}

\newtheorem{theorem}{Theorem} [section]
\newtheorem{lemma}[theorem]{Lemma}

\newtheorem{cor}[theorem]{Corollary}
\newtheorem{conjecture}[theorem]{Conjecture}
\def\N{\mathbb N}
\def\R{\tilde R}
\def\RR{\text{RR}}
\def\RRC{\text{RRC}}
\def\rr{\tilde{r}}
\def\sr{\hat{r}}

\begin{document}

\title{Online size Ramsey numbers: Path vs $C_4$}
\author{Grzegorz Adamski}
\author{Ma\l gorzata Bednarska-Bzd\c ega}
\affil{Faculty of Mathematics and CS, Adam Mickiewicz University, Pozna\'n, Poland}

\maketitle

\begin{abstract}
Given two graphs $G$ and $H$, a size Ramsey game is  
played on the edge set of $K_\N$. In every round, Builder selects an edge and Painter colours 
it red or blue.  Builder's goal is to force Painter to create a red copy of $G$ or a blue copy of $H$ as soon as possible. The online (size) Ramsey number $\rr(G,H)$ is the number of rounds in the game provided  Builder and Painter play optimally. 
We prove that $\rr(C_4,P_n)\le 2n-2$ for every $n\ge 8$.
The upper bound matches the lower bound obtained by J.~Cyman, T.~Dzido, J.~Lapinskas, and A.~Lo, so we get $\rr(C_4,P_n)=2n-2$ for $n\ge 8$.
Our proof for $n\le 13$ is computer assisted. 
The bound $\rr(C_4,P_n)\le 2n-2$ solves also the ``all cycles vs. $P_n$'' game for $n\ge 8$ -- it implies that it takes Builder $2n-2$ rounds to force Painter to create a blue path on $n$ vertices or any red cycle.   
\end{abstract}

\section{Introduction}

Let $G$ and $H$ be finite graphs. Consider the following game  $\R(G,H)$ played on the infinite board $K_\N$ (i.e. the board is a complete graph with the vertex set $\N$). In every round, Builder chooses a previously unselected edge of $K_\N$ and Painter colours it red or blue. The game ends when there is a red copy of $G$ or a blue copy of a $H$ at the board.  The goals of the players are opposite, Builder aims to finish the game as soon as possible. The game, for $H=G=K_n$, was introduced by Beck \cite{beck}. 

By $\rr(G,H)$ we denote the number of rounds in the game $\R(G, H)$, provided both players play optimally and call it the  online size Ramsey number. Let us emphasise that in the literature online size Ramsey numbers are called also online Ramsey numbers. The online size Ramsey numbers  $\rr(G, H)$ are less intensively studied than classic Ramsey numbers or size Ramsey numbers $\sr(G, H)$, i.e.~the minimum number of edges in a graph with the property that every two-colouring of its edges results in a red copy of $G$ or a blue copy of $H$. There are very few exact results for $\rr(G, H)$, even if every graph $G$, $H$ is a path or a cycle. It is known that $\rr(P_3,P_n)=\lceil 5(n-1)/4\rceil$ and $\rr(P_3,C_n)=\lceil 5n/4\rceil$ \cite{lo}. Nonetheless, there are quite a few papers where paths and cycles are involved, for example \cite{ga}, \cite{gryt}, \cite{ind}, \cite{lo}, \cite{ms}. There are also computer assisted results on selected graphs on up to $9$ vertices: \cite{pral}, \cite{pral1}, \cite{pral2}. 

In this paper we study the game $\R(C_4,P_n)$. As for small paths, it is known that: $\rr(C_4,P_3)=6$, $\rr(C_4,P_4)=8$ (\cite{lo}), $\rr(C_4,P_5)=9$ (\cite{dzido}), and $\rr(C_4,P_6)=11$ (\cite{ml}). In general, we have $\rr(C_4,P_n)\ge 2n-2$ for $n\ge 2$ and it follows from a stronger result by Cyman, Dzido, Lapinskas and Lo \cite{lo}, who proved that $\rr(C_k,H)\ge |V(H)|+|E(H)|-1$  for every $k\ge 3$ and any connected graph $H$. The best upper bound known so far for $n\ge 7$ is $\rr(C_4,P_n)\le 3n-5$ by  Dybizba\'nski, Dzido and Zakrzewska (\cite{dzido}). We improve the upper bound and this is the main result of our paper. 

\begin{theorem}\label{main}
$\rr(C_4,P_n)\le 2n-2$ for every $n\ge 8$. 
Furthermore, $\rr(C_4,P_7)\le 13$.
\end{theorem}

In view of the above mentioned lower bound $\rr(C_4,P_n)\ge 2n-2$ as well as the known bound $\rr(C_4,P_7)\ge 13$  (\cite{dzido}), our theorem implies the following new exact results. 

\begin{cor}
If $n\ge 8$, then $\rr(C_4,P_n)=2n-2$. 
Furthermore, $\rr(C_4,P_7)=13$.
\end{cor}

Theorem \ref{main} solves also another problem related to games with paths and cycles. Schudrich \cite{ms} studied a Builder-Painter game on $K_\N$ in which Builder wants to force Painter to create a red cycle or a blue path $P_n$ as soon as possible, while Painter wants to keep the red graph acyclic and the blue graph free of $P_n$ as long as possible. Let $\rr({\mathcal C},P_n)$ denote the number of rounds provided Builder and Painter play optimally. 
Schudrich proved that $\rr({\mathcal C}, P_n)\le 2.5 n+5$ for every $n$.  Theorem \ref{main} improves this bound for $n\ge 8$ and, in view of the bound $\rr({\mathcal C}, P_n)\ge 2n-2$ obtained by the authors of \cite{lo}, we get the following result.

\begin{cor}
If $n\ge 8$, then $\rr({\mathcal C}, P_n)= 2n-2$. 
\end{cor}

It is interesting that the corresponding size Ramsey number $\sr({\mathcal C}, P_n)$ asymptotically does not differ very much from $\rr({\mathcal C}, P_n)$ in sense of the multiplicative constants. More precisely, in view of Theorem \ref{main} and the bounds on $\sr({\mathcal C}, P_n)$ proved by Bal and Schudrich \cite{bal}, we have
$$
1.03+o(1)\le \frac{\sr({\mathcal C}, P_n)}{\rr({\mathcal C}, P_n)}\le 1.98 +o(1).
$$
As the authors of \cite{bal} note, the upper bound $\sr({\mathcal C}, P_n)< 3.947 n +O(1)$ has a computer assisted proof. 

Theorem \ref{main} implies that in sense of the multiplicative constant the numbers $\rr(C_4,P_n)=2n-2$ are separated from $\rr(C_{2k+1},P_n)=2n-2$ for any fixed $k$. We know that Painter can avoid any odd red cycle and a blue $P_n$ for more than $2.6n-4$ rounds \cite{ga}. It is not surprising that it is easier to avoid an odd cycle than an even cycle and we suspect that for Builder every game $\R(C_{2k},P_n)$ is easier than, say, $\R(C_3,P_n)$. Therefore we conjecture that asymptotically the numbers $\rr(C_{2k},P_n)$ do not differ much from $\rr(C_{4},P_n)$. More precisely, we pose the following conjecture. 

\begin{conjecture}
$\rr(C_{2k},P_n)=2n+o(n)$ for every fixed $k\ge 3$. 
\end{conjecture}

The proof of Theorem \ref{main} is based on induction. Our inductive argument, technically quite complicated, holds for $n\ge 14$ and two next sections are devoted mainly to that argument. 
For $7\le n\le 13$ we analyse the games with the computer help. The idea of the algorithm is described in Section \ref{algorithm} and the C++ code is enclosed in the appendix.

\section{Preliminaries}\label{prelim}

For a graph $H$ we put $v(H)=|V(H)|$, $e(H)=|E(H)|$. If $A\subs V(H)$, then we denote by $N_H(A)$ the set of all edges of $H$ with at least one end in the set $A$ and we put $\deg_H(u)=|N_H(\{u\}|$.

We say that a graph $H$ is coloured if every its edge is blue or red. A graph is red (or blue) if every its edge is red (blue). 
Let $c_1,c_2,\ldots,c_k\in\{b,r\}$ be consecutive edge colours of a coloured path $P_{k+1}$. 
Then the coloured path is called a $c_1c_2\ldots c_k$-path. 
We say a vertex of $K_\N$ is free at a moment of the game if it is not incident to any edge selected by Builder so far.

Given $n\ge 2$, and a coloured graph $H$ (it may be empty), consider the following auxiliary game $\RR(C_4,P_n,H)$. 
The board of the game is 
$K_\N$, with exactly $e(H)$ edges coloured, inducing a copy of $H$ in $K_\N$. The rules of selecting and colouring edges
by Builder and Painter are the same as in the standard game, however, Painter is not allowed to colour an edge red if that 
would create a red $C_4$. Builder wins if $e(H)\le 2n-2$ and after at most $2n-2-e(H)$ rounds 
there is  a blue $P_n$ at the board and exactly $n-1$ edges are blue; otherwise Painter is the winner. 
The game ends after $2n-2-e(H)$ rounds or sooner; we assume that  it stops before the round $2n-2-e(H)$ 
if more than $n-1$ edges are blue or there is a blue $P_n$.   If $H$ is an empty graph then we write 
$\RR(C_4,P_n)$ instead of $\RR(C_4,P_n,H)$.  

It is clear that the first part of Theorem \ref{main} follows immediately from the following theorem.

 \begin{theorem}\label{main2}
Suppose that $G$ is one of the coloured graphs: an empty graph, a $b$-path, a $br$-path, a $brb$-path, a $brr$-path, or a $brrb$-path.
Then for every $n\ge 8$ Builder has a winning strategy in $\RR(C_4,P_n,G)$. 
\end{theorem}

We will prove the above theorem in the next section. In the remaining part of this section, we prepare tools for the inductive argument of the proof. We begin by a lemma need for the induction base cases. 

\begin{lemma}\label{warpocz}
Suppose that $H$ is one of the coloured graphs: a $b$-path, a $br$-path, a $brb$-path, a $brr$-path, or a $brrb$-path.
Then for every $7\le n\le 13$ Builder has a winning strategy in $\RR(C_4,P_n,H)$.
Furthermore,  for every $8\le n\le 13$ Builder has a winning strategy in $\RR(C_4,P_n)$.
\end{lemma}

The proof of this lemma is based on computer calculations. We present the algorithm and its analysis in Section \ref{algorithm}. 

For the analysis of the game $\RR(C_4,P_n,H)$ we need some additional definitions. 
We say that Builder in $\RR(C_4,P_n,H)$ forces a blue path of length $k\ge 1$ if for $k$ consecutive rounds he selects
edges such that must be coloured blue according to the rules of the game and the $k$ edges induce a path of length $k$. 

Suppose that $n,k\in\N$, $H$ is a coloured graph and contains a blue path $u_0u_1\ldots u_{k+1}$. We say that $H'$ is the graph obtained by contraction of the path $u_0u_1\ldots u_{k+1}$ to the edge $u_0u_{k+1}$
if $V(H')=V(H)\setminus \{u_1,u_2,\ldots,u_k\}$, $E(H')=(E(H)\setminus N_H(\{u_1,u_2,\ldots,u_k\}))\cup\{u_0u_{k+1}\}$, 
the edge $u_0u_{k+1}$ is blue (if it was red in $H$, we change its colour) and the colours of the remaining edges in $H'$ are the same as in $H$. 
For $x,y\in V(H)$ we define $\delta_H(x,y)=1$ if $x,y$ are adjacent in $H$ and $\delta_H(u_0,u_{k+1})=0$ otherwise.

The following lemma is the key realisation of the inductive argument used in the next section. 

\begin{lemma}\label{sciag}
Suppose that $n,k\in\N$, $H$ is a coloured graph such that it contains a blue path $P=u_0u_1\ldots u_{k+1}$ and $u_0u_{k+1}\notin E(H)$ or $u_0u_{k+1}$ is red. 
Furthermore, suppose that no edge from the set $N_H(\{u_1,u_2,\ldots,u_k\})\setminus E(P)$ is blue and
$$|N_H(\{u_1,u_2,\ldots,u_k\})|+\delta_H(u_0,u_{k+1})\le 2k+1.$$
Let $H'$ be the graph obtained from $H$ by contraction of the path $P$ to the edge $u_0u_{k+1}$. 
If Builder has a winning strategy in $\RR(C_4,P_{n-k},H')$, then he has a winning strategy in $\RR(C_4,P_n,H)$.
\end{lemma}

\begin{proof}
Let $H$, $P$ and $H'$ be as in the assumption of the lemma. 
Suppose that $K_\N$ contains the coloured graph $H$ and $K'_\N$ is the complete graph arising after the 
contraction of the path $P=u_0u_1\ldots u_{k+1}$ to the edge $u_0u_{k+1}$. Consider two games:
$\RR(C_4,P_n,H)$ played on $K_\N$ and $\RR(C_4,P_{n-k},H')$ played on $K'_\N$. 
We assumed that no edge from the set $N_H(\{u_1,u_2,\ldots,u_k\})\setminus E(P)$ is blue and $u_0u_k$ is not blue at start of 
the game $\RR(C_4,P_n,H)$ so at the beginning of both games the sets $B_0$ and $B'_0$ of all blue edges in 
$\RR(C_4,P_n,H)$ and $\RR(C_4,P_{n-k},H')$ respectively, satisfy the condition $B_0=(B'_0\setm\{u_0,u_{k+1}\})\cup E(P)$
and $|B_0|=|B'_0|-1+k+1=|B'_0|+k$.

Builder in $\RR(C_4,P_n,H)$ uses Builder's winning strategy in $\RR(C_4,P_{n-k},H')$ 
by selecting the same edge as  Builder in $\RR(C_4,P_{n-k},H')$ in every round. 
Observe that this way Builder in $\RR(C_4,P_n,H)$ never selects $u_0u_{k+1}$ nor any edge incident to any of 
vertices $u_1,u_2,\ldots,u_k$.  
We assumed that Builder wins $\RR(C_4,P_{n-k},H')$ so after at most $2(n-k-1)-e(H')$ rounds of this game 
the set $B'$ of all blue edges has exactly $n-k-1$ elements and there is a blue path $P'$ with $n-k-1$ edges. 
Thus $E(P')=B'$ and it contains the blue edge $u_0u_{k+1}$. 
In the corresponding game $\RR(C_4,P_n,H)$ after at most $2(n-k-1)-e(H')$ rounds
the set $B$ of blue edges at the board differs from $B'$ in the same way as $B_0$ differs from $B'_0$ so 
$B=(B'\setm\{u_0,u_{k+1}\})\cup E(P)$ and $|B|=|B'|+k$. We infer that $|B|=n-1$ and the blue edges  $(E(P')\setm\{u_0,u_{k+1}\})\cup E(P)$ form a path on $n$ vertices. 

Finally, let us estimate the number of rounds. If follows from the contraction definition that 
$$e(H')=e(H)-|N_H(\{u_1,u_2,\ldots,u_k\})|+1-\delta_H(u_0,u_{k+1}).$$
Thus, in view of the assumptions of the lemma, we verify that
$$e(H')\ge e(H)-(2k+1)+1=e(H)-2k.$$
We know that the number of rounds in both games is the same and not greater than 
$2(n-k-1)-e(H')$, so the game $\RR(C_4,P_n,H)$ lasts not more than
$$2(n-k-1)-e(H')\le 2(n-k-1)-e(H)+2k=2n-2-e(H)$$
rounds.
We conclude that Builder wins the game $\RR(C_4,P_n,H)$.
 \end{proof}

We need two more lemmata, about graphs called butterflies. 
The following graph will be called an $(s,t)$-butterfly with centers $(u_0,u_1)$.

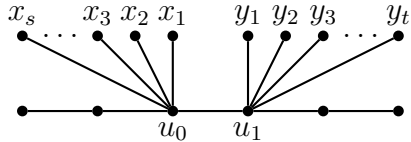
\begin{center}
 \begin{tikzpicture}
\jezk{6}{3}{4};
\node [below] at (a3) {$u_0$};
\node [below] at (a4) {$u_1$};
\node [above] at (b1) {$x_1$};
\node [above] at (b2) {$y_1$};
\wierz{x2}{$(b1)-(0.5,0)$};
\wierz{x3}{$(b1)-(1.0,0)$};
\wierz{xs}{$(b1)-(2.0,0)$};
\node [above] at (x2) {$x_2$};
\node [above] at (x3) {$x_3$};
\node [above] at (xs) {$x_s$};
\node at ($(x3)-(0.5,0)$) {$\ldots$};
\krk{a3}{x2};\krk{a3}{x3};\krk{a3}{xs};
\wierz{y2}{$(b2)+(0.5,0)$};
\wierz{y3}{$(b2)+(1.0,0)$};
\wierz{yt}{$(b2)+(2.0,0)$};
\node [above] at (y2) {$y_2$};
\node [above] at (y3) {$y_3$};
\node [above] at (yt) {$y_t$};
\node at ($(y3)+(0.5,0)$) {$\ldots$};
\krk{a4}{y2};\krk{a4}{y3};\krk{a4}{yt};
 \end{tikzpicture}
\end{center}

The subgraph of the butterfly, consisting of edges $u_0x_i$ with $i=1,2,\ldots,s$ and the subgraph consisting of edges $u_1y_i$ with $i=1,2,\ldots,t$, are called the $u_0$-wing and $u_1$-wing of the butterfly, respectively.

\begin{lemma}\label{motyle}
Let $s\ge 1$ and $H$ be a red $(s,s)$-butterfly  with centers $(u_0,u_1)$. 
Then:
\begin{enumerate}[label=(\roman*)]
\item \label{bb}
 Builder in $\RR(C_4,P_n,H)$ can force a blue path on the vertex set $V(H)$, with ends $u_0,u_1$.
\item \label{br}
 If $s\ge 2$, Builder in $\RR(C_4,P_n,H)$ can force a blue path on the vertex set $V(H)\setm\{u_1\}$,
with ends $u_0,y$ for a vertex $y$ of the $u_1$-wing. 
\item \label{rr}
If $s\ge 2$, Builder in $\RR(C_4,P_n,H)$ can force a blue path on the vertex set $V(H)\setm\{u_0,u_1\}$,
with ends $x,y$ for a vertex $x$ of the $u_0$-wing and a vertex $y$ of the $u_1$-wing. 
\end{enumerate}
\end{lemma}

\begin{proof}
Let us consider the following red $(s,s)$-butterfly $H$, with $s\ge 1$.

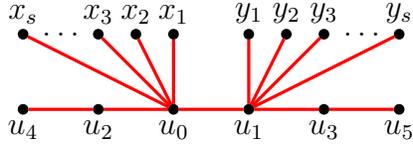
\begin{center}
 \begin{tikzpicture}
\jez{6}{3}{4};
\node [below] at (a3) {$u_0$};
\node [below] at (a4) {$u_1$};
\node [below] at (a2) {$u_2$};
\node [below] at (a5) {$u_3$};
\node [below] at (a1) {$u_4$};
\node [below] at (a6) {$u_5$};
\node [above] at (b1) {$x_1$};
\node [above] at (b2) {$y_1$};
\wierz{x2}{$(b1)-(0.5,0)$};
\wierz{x3}{$(b1)-(1.0,0)$};
\wierz{xs}{$(b1)-(2.0,0)$};
\node [above] at (x2) {$x_2$};
\node [above] at (x3) {$x_3$};
\node [above] at (xs) {$x_s$};
\node at ($(x3)-(0.5,0)$) {$\ldots$};
\kr{a3}{x2};\kr{a3}{x3};\kr{a3}{xs};
\wierz{y2}{$(b2)+(0.5,0)$};
\wierz{y3}{$(b2)+(1.0,0)$};
\wierz{yt}{$(b2)+(2.0,0)$};
\node [above] at (y2) {$y_2$};
\node [above] at (y3) {$y_3$};
\node [above] at (yt) {$y_s$};
\node at ($(y3)+(0.5,0)$) {$\ldots$};
\kr{a4}{y2};\kr{a4}{y3};\kr{a4}{yt};
 \end{tikzpicture}
\end{center}

Let $V_1=\{x_1,x_2,...x_s,u_2,u_5\}$ and $V_2=\{y_1,y_2,...y_s,u_3,u_4\}$. Consider (uncoloured) graph $G$ on the vertex set $V(H)\setminus\{u_0,u_1\}$ such that two vertices $w,w'$ are adjacent in $G$ if and only if there is a path of length 3 in $H$, with ends $w,w'$. This is not hard to observe that $G$ is a copy of the graph $K_{s+2,s+2}-P_4$ such that the sets  $V_1$ and $V_2$ are the parts of its bipartition. $G$ is obtained from the complete bipartite graph $K_{s+2,s+2}$ by deleting three edges of the path  $u_2u_4u_5u_3$.

\begin{figure}[h]	
\centering	
 \begin{tikzpicture}	
\wierz{x1}{$(0,0)$};	
\node [below] at (x1) {$x_1$};	
\wierz{y1}{$(0,1)$};	
\node [above] at (y1) {$y_1$};	
\node at ($(0.5,0)$) {$\ldots$};	
\node at ($(0.5,1)$) {$\ldots$};	
\wierz{xs}{$(1,0)$};	
\node [below] at (xs) {$x_s$};	
\wierz{ys}{$(1,1)$};	
\node [above] at (ys) {$y_s$};	
\wierz{u2}{$(2,0)$};	
\node [below] at (u2) {$u_2$};	
\wierz{u3}{$(2,1)$};	
\node [above] at (u3) {$u_3$};	
\wierz{u5}{$(3,0)$};	
\node [below] at (u5) {$u_5$};	
\wierz{u4}{$(3,1)$};	
\node [above] at (u4) {$u_4$};	
\wierz{u1}{$(4,0)$};	
\node [below] at (u1) {$u_1$};	
\wierz{u0}{$(4,1)$};	
\node [above] at (u0) {$u_0$};	
\krk{x1}{y1};	
\krk{x1}{ys};	
\krk{x1}{u3};	
\krk{x1}{u4};	
\krk{xs}{y1};	
\krk{xs}{ys};	
\krk{xs}{u3};	
\krk{xs}{u4};	
\krk{u2}{y1};	
\krk{u2}{ys};	
\krk{u2}{u3};	
\krk{u5}{y1};	
\krk{u5}{ys};	
\krk{u5}{u0};	
\krk{u1}{u4};	
%
%
 \end{tikzpicture}	
\caption{Let $J$ be the graph of all edges that Builder can force to be blue in $H$. Lemma \ref{motyle} is equivalent to fact that if $s\ge s_0$, then $X$ has a hamiltonian path from $a$ to $b$ where $(X,a,b,s_0)\in \{(J,u_0,u_1,1),$ $(J-\{u_1\},u_0,y_i,2),(J-\{u_0,u_1\},x_i,y_j,2)\}$.}	
\label{tab:H3}	
\end{figure}
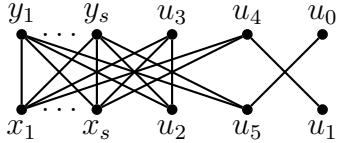	
One can verify that if $s\ge 2$, then for every $w_1\in V_1, w_2\in V_2$ there is a hamiltonian path with ends $w_1, w_2$ in $G=K_{s+2,s+2}-P_4$ (in other words, $G$ is Hamilton-laceable). Furthermore, if $s=1$, then there is a hamiltonian path with ends $u_4, u_5$ in $G$.

Note that every edge $e\in E(G)$, if coloured red, creates a red $C_4$ in $H\cup\{e\}$.
Thus if Builder selects all edges of a hamiltonian path $P\subseteq G$ in the game $\RR(C_4, P_n, H)$, then he forces a blue path on the vertex set $V(H)\setminus\{u_0,u_1\}$. 

It follows from the above observations that
Builder can force a blue path on the vertex set $V(H)\setminus\{u_0,u_1\}$, with ends $u_4,u_5$ and, if $s\ge 2$, then he can force a blue path on the vertex set $V(H)\setminus\{u_0,u_1\}$, with any ends 
$w_1\in V_1, w_2\in V_2$. We will use this fact in all three parts of the lemma.

Part \ref{rr} is an immediate consequence of the above observation and the fact that vertices $x_i$ and $y_i$ are in different bipartition sets.   

In order to prove \ref{bb}, we consider $s\ge 1$ and a blue path $P$ on the vertex set $V(H)\setminus\{u_0,u_1\}$, with ends $u_4,u_5$, forced by Builder. Then Builder selects edges $u_5u_0$ and $u_4u_1$. Painter has to colour them blue, otherwise a red $C_4$ is created. Then $P$ extended by $u_5u_0$, $u_4u_1$ forms a blue path on the vertex set $V(H)$, with ends $u_0,u_1$.

In order to prove \ref{br}, we consider $s\ge 2$ and a blue path $P$ on the vertex set $V(H)\setminus\{u_0,u_1\}$, with ends $u_5\in V_1$ and $y_s\in V_2$, forced by Builder. Then Builder selects the edge $u_5u_0$ and Painter has to colour it blue.
Thereby a blue path on the vertex set $V(H)\setminus\{u_1\}$ is formed, with ends $u_0,y_s$.

\end{proof}

\begin{lemma}\label{malo}
Let $s\ge 2$, $s'\in\{s,s+1\}$ and $H$ be a coloured $(s',s)$-butterfly with centers $(c_1,c_2)$. 
Suppose that every wing has at least two red edges, exactly one or two edges in the $c_1$-wing are blue
and the $c_2$-wing has at most as many blue edges as the $c_1$-wing.    
Furthermore, if every wing has two blue edges, then $s'=s$.
Then Builder has a winning strategy in $\RR(C_4,P_{v(H)},H)$.
\end{lemma}

\begin{proof} 

Given $u\in V(H)$, let $\deg_r(u)$ and $\deg_b(u)$  be the number of red and blue edges of $H$, respectively, incident to $u$. Note that $\deg_r(u)+\deg_b(u)=\deg_H(u)$. It is not hard to verify that if $H$ fulfills the assumption of the lemma, then $\deg_H(c_1)=s'+2$, $\deg_H(c_2)=s+2$, and it satisfies one
of the following conditions.
\begin{enumerate}[label=(\roman*)]
\item \label{malobb}
$\deg_b(c_1)=1, \deg_b(c_2)\le 1$ and $\deg_r(c_1)=\deg_r(c_2)$; 
\item \label{malorr}
$\deg_b(c_1)=2, \deg_b(c_2)\in \{1,2\}$ and $\deg_r(c_1)=\deg_r(c_2)$; 
\item \label{malorb}
$\deg_b(c_1)\in\{1,2\}, \deg_b(c_2)\le 1$ and $\deg_r(c_1)=\deg_r(c_2)-1$;
 \item \label{malobb2}
$\deg_b(c_1)=\deg_b(c_2)=1$ and $\deg_r(c_1)=\deg_r(c_2)+1$;
\item \label{malorb2}
$\deg_b(c_1)=2, \deg_b(c_2)=0$ and $\deg(c_1)=\deg(c_2)$.
\end{enumerate}

In order to verify that the above cases cover all possible values of $s'-s=\deg(c_1)-\deg(c_2)$, $\deg_b(c_1)$, $\deg_b(c_2)$, we summarize them in Table \ref{tab:cases}.

\begin{table}[h]
\centering
\begin{tabular}{l | l | l | l}
$s'-s$ & $\deg_b(c_1)$ & $\deg_b(c_2)$ & case\\
\hline
0&1&0& \ref{malorb}\\
0&1&1& \ref{malobb}\\
0&2&0& \ref{malorb2}\\
0&2&1& \ref{malorb}\\
0&2&2& \ref{malorr} \\
1&1&0& \ref{malobb} \\
1&1&1& \ref{malobb2}\\
1&2&0& \ref{malorb}\\
1&2&1& \ref{malorr}
\end{tabular}
\caption{All possible values of $s'-s,\deg_b(c_1),\deg_b(c_2)$ in Lemma \ref{malo}.}
\label{tab:cases}
\end{table}

We consider every case separately. In every case we denote by $P(c_i)$ the path (of length at most two, possibly trivial) 
induced by all blue edges incident to $c_i$, for $i=1,2$ and $a_i,b_i$ denote the ends of $P(c_i)$. 
We assume that $b_1\neq c_1$ and, if $P(c_2)$ is nontrivial,  then $b_2\neq c_2$.  Let $v(H)=n$. 
Our goal is to present Builder's strategy such that after at most $2n-2-e(H)=n-1$ rounds of $\RR(C_4,P_{v(H)},H)$
there is a blue path on $n$ vertices and it contains all blue edges of the board.
By $B$ we denote the red butterfly spanned by all red edges of $H$.

In case \ref{malobb} we apply Lemma \ref{motyle}\ref{bb} to the butterfly $B$ so Builder in $\RR(C_4,P_{n},H)$, 
as Builder in $\RR(C_4,P_{v(B)},B)$, 
can force a blue path $P$ with $V(P)=V(B)$, with ends $c_1,c_2$. 
This path, extended by $P(c_1)$ and $P(c_2)$ (both of length at most one) forms a blue path on $n$ vertices
and clearly contains all blue edges of the board. Builder achieves that goal within $e(P)\le e(H)=n-1$ rounds
so he wins the game.

In case \ref{malorr} let $x_i,y_i$ be any two vertices of the $c_i$-wing of the butterfly $B$, for $i=1,2$. 
Builder starts with selecting the edge $b_1x_1$ and, if Painter colours  it red, then Builder selects
$b_1y_1$ and it must be coloured blue. So after at most two rounds exactly one of the edges $b_1x_1$,
$b_1y_1$ is blue. Builder continues in the same way  with edges $b_2x_2$, $b_2y_2$ and obtains
another blue edge. Without loss of generality, we assume that edges $b_1y_1$ and $b_2y_2$ are blue.
At most 4 rounds are made in the game so far.  
Then we apply Lemma \ref{motyle}\ref{rr} to the butterfly $B$. Hence Builder in $\RR(C_4,P_{n},H)$, 
as Builder in $\RR(C_4,P_{v(B)},B)$, 
can force a blue path $P$ with $V(P)=V(B)\setm\{c_1,c_2\}$, with ends $y_1,y_2$. 
The edge set 
$$E(P(c_1))\cup\{b_1y_1\}\cup E(P)\cup\{b_2y_2\}\cup E(P(c_2))$$
forms a blue path on the vertex set $V(H)$ and this path contains all blue edges of the board. 
Builder achieves that goal within at most 
$$4+e(P)=3+v(B)-2=2+e(B)=2+e(H)-e(P(c_1))-e(P(c_2))\le e(H)=n-1$$ 
rounds so he wins the game.
 
Consider case \ref{malorb}.
Let $x,y$ be any two vertices of the $c_2$-wing of the butterfly $B$. 
Builder starts with selecting the edge $b_1x$ and, if Painter colours  it red, then Builder selects
$b_1y$ and it must be coloured blue. So after at most two rounds exactly one of the edges $b_1x$,
$b_1y$ is blue and we assume that $b_1y$ is blue.   Let $B'=B\setm\{c_2y\}$. Then $B'$
is an $(s,s)$-butterfly with $s\ge 2$. 
Based on Lemma \ref{motyle}\ref{br},  Builder in $\RR(C_4,P_{v(B')},B')$ 
-- and hence also Builder in $\RR(C_4,P_{n},H)$ -- can force a blue path $P$ with $V(P)=V(B')\setm\{c_1\}$, 
with ends $c_2,z$, with some vertex $z$ of the $c_1$-wing. 
Builder forces such a path, then selects $yz$. Painter has to colour it blue since otherwise a red cycle $c_1c_2yz$ 
would have appeared. Observe that  the edge set $E(P(c_2))\cup E(P)\cup \{zy, yb_1\}\cup E(P(c_1))$ 
forms a blue path on the vertex set $V(H)$ and this path contains all blue edges of the board. 
Builder achieves that goal within at most 
$$2+e(P)+1=2+v(B')-1=2+e(B')=2+e(H)-1-e(P(c_1))-e(P(c_2))\le e(H)=n-1$$ 
rounds so he wins the game.

The analysis in case \ref{malobb2} is similar to the previous one so we omit it.

We proceed to the last case. Here $P(c_1)$ is a $bb$-path with ends $a_1,b_1\neq c_1$ and $P(c_2)$ is trivial.
Furthermore, the butterfly $B$ spanned by all red edges of $H$ has its
$c_2$-wing bigger by two than its $c_1$-wing.
The strategy of Builder is similar to the one in case \ref{malorb}. 
Builder begins by obtaining a blue edge $b_1y$ for a vertex $y$ of the $c_2$-wing of the butterfly $B$
and it takes him at most two rounds. Since in case \ref{malorb2} the $c_2$-wing has at least four edges, 
Builder can get another blue edge $a_1x$ for a vertex $x\neq y$ of the $c_2$-wing, within 
at most two rounds. 
Let $B'=B\setm\{c_2x, c_2y\}$. Then $B'$ is an $(s,s)$-butterfly with $s\ge 2$. 
We apply Lemma \ref{motyle}\ref{br} and Builder in $\RR(C_4,P_{n},H)$ forces 
a blue path $P$ with $V(P)=V(B')\setm\{c_1\}$, 
with ends $c_2,z$, with some vertex $z$ of the $c_1$-wing. 
Then Builder selects $yz$ and Painter has to colour it blue. 
Now the edge set $E(P)\cup \{zy, yb_1\}\cup E(P(c_1))\cup \{a_1x\}$ 
forms a blue path on the vertex set $V(H)$ and this path contains all blue edges of the board. 
Builder achieves that goal within at most 
$$4+e(P)+1=4+e(B')=4+e(H)-2-e(P(c_1))= e(H)=n-1$$ rounds. 

In all five cases Builder wins the game $\RR(C_4,P_{n},H)$.
\end{proof}

\section{Proof of theorem  \ref{main2}}

We prove the theorem by induction on $n$. If $8\le n\le 13$, the assertion follows from  Lemma \ref{warpocz}.

Suppose that $n\ge 14$ and assume that the theorem is true for the game $\RR(C_4,P_{n'},G)$ for every $n'< n$
and every $G$ described in the theorem. We begin by analysis of the game $\RR(C_4,P_{n})$, i.e. in case $G$ is empty.
 
\subsection{$G$ is empty}\label{gpusty}

The strategy of Builder depends on the moment the first blue edge is coloured by Painter. 
As long as Painter colours all edges red, Builder selects the edges in the following order.

\begin{itemize}
\item Stage 1.

First, the edges of a $(1,1)$-butterfly with centers $(u_0,u_1)$ are selected in order, given by edge labels:

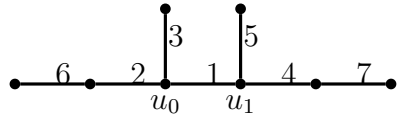
\begin{center}
\begin{tikzpicture}
\foreach \X in {1,...,6}{\wierz{a\X}{$(\X-1,0)$};}
\wierz{b1}{$(a3)+(0,1)$};
\wierz{b2}{$(a4)+(0,1)$};
\krke{a1}{a2}{6};
\krke{a2}{a3}{2};
\krke{a3}{a4}{1};
\krke{a4}{a5}{4};
\krke{a5}{a6}{7};
\krke{a3}{b1}{3};
\krke{a4}{b2}{5};
\node [below] at (a3) {$u_0$};
\node [below] at (a4) {$u_1$};
 \end{tikzpicture}
\end{center}

\item Stage 2.

Then (as long as Painter colours all edges red) in every even round $j$ ($j\ge 8$) Builder selects an edge incident to the butterfly center $u_0$
and any free vertex, denoted by $u_j$, while in every odd round $j$ ($j\ge 9$) he selects an edge incident to another butterfly center $u_1$ and any free vertex $u_j$. 
This stage ends if either $n$ is even and we have $n-1$ coloured edges, or $n$ is odd and we have $n-2$ coloured edges. 
Thus,  for even $n$ after Stage 2, if all edges are red, they induce an $((n-6)/2,(n-6)/2)$-butterfly:

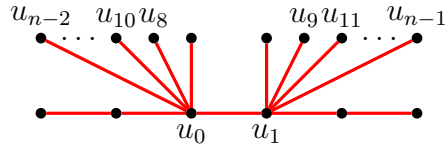
\begin{center}
 \begin{tikzpicture}
\motyl;
\node [below] at (a3) {$u_0$};
\node [below] at (a4) {$u_1$};
\node [above] at (u8) {$u_8$};
\node [above] at (u10) {$u_{10}$};
\node [above] at (upost) {$u_{n-2}$};
\node [above] at (u9) {$u_9$};
\node [above] at (u11) {$u_{11}$};
\node [above] at (uost) {$u_{n-1}$};
 \end{tikzpicture}
\end{center}

In case of odd $n$, if after Stage 2 all edges are red, they induce an $((n-7)/2,(n-7)/2)$-butterfly:

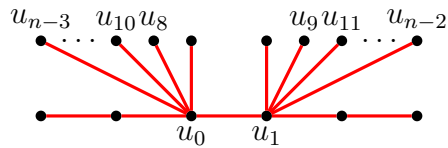
\begin{center}
 \begin{tikzpicture}
\motyl;
\node [below] at (a3) {$u_0$};
\node [below] at (a4) {$u_1$};
\node [above] at (u8) {$u_8$};
\node [above] at (u10) {$u_{10}$};
\node [above] at (upost) {$u_{n-3}$};
\node [above] at (u9) {$u_9$};
\node [above] at (u11) {$u_{11}$};
\node [above] at (uost) {$u_{n-2}$};
 \end{tikzpicture}
\end{center}

\item Stage 3.

This stage only takes place if $n$ is odd (and all edges are red). Builder selects an edge incident to $u_{n-2}$ and any 
free vertex $u_{n-1}$. Thus,  after Stage 3, if all edges are red, they induce a $((n-7)/2,(n-7)/2)$-butterfly with the additional edge $u_{n-2}u_{n-1}$:

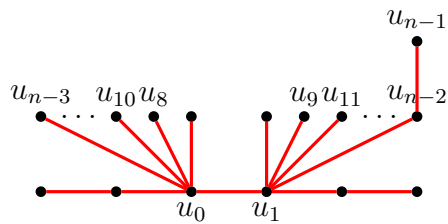
\begin{center}
 \begin{tikzpicture}
\motyl;
\node [below] at (a3) {$u_0$};
\node [below] at (a4) {$u_1$};
\node [above] at (u8) {$u_8$};
\node [above] at (u10) {$u_{10}$};
\node [above] at (upost) {$u_{n-3}$};
\node [above] at (u9) {$u_9$};
\node [above] at (u11) {$u_{11}$};
\node [above] at (uost) {$u_{n-2}$};
\wierz{un}{$(uost)+(0,1)$};
\node [above] at (un) {$u_{n-1}$};
\kr{uost}{un};
 \end{tikzpicture}
\end{center}

\end{itemize}

We will analyse the play based on Painter choices; Builder's strategy depends on the time the first blue edge appears (if any) in Stages 1, 2, 3.

\subsubsection{No blue edges in Stages 1, 2 and 3}

If $n$ is even and no blue edges after round $n-1$,
then the $n-1$ edges induce a red $(s,s)$-butterfly on $n$ vertices and in view of Lemma \ref{motyle}\ref{bb}
Builder can force a blue path $P_n$. After $2n-2$ rounds $n-1$ edges are red and $n-1$ edges are blue.

If $n$ odd and no blue edges after round $n-1$, 
then we have a red $(s,s)$-butterfly on $n-1$ vertices, with centers $u_0$, $u_1$ and a red edge $u_{n-2}u_{n-1}$ as
described in Stage 3. Then  in view of Lemma \ref{motyle}\ref{bb}  Builder can force
a blue path $P_{n-1}$ on the vertex set of the $(s,s)$-butterfly, with ends $u_0$, $u_1$. Then he selects
the edge $u_0u_{n-1}$ and it must be coloured blue since vertices $u_0,u_{n-1}$ are ends of a red path of length 3.
This edge extends a blue path $P_{n-1}$.
Thus after $2n-2$ rounds $n-1$ edges are red and $n-1$ edges are blue and the blue edges induce a path $P_n$. 

In both cases Builder wins $\RR(C_4,P_{n})$.

\subsubsection{First blue edge in Stage 3}

Then $n\ge 14$ is odd, and after $n-1$ rounds we have the following coloured graph $H$ at the board.

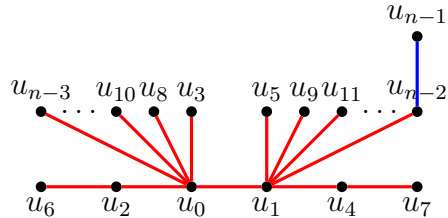
\begin{center}
 \begin{tikzpicture}
\motyl;
\node [below] at (a3) {$u_0$};
\node [below] at (a4) {$u_1$};
\node [below] at (a2) {$u_2$};
\node [above] at (b1) {$u_3$};
\node [below] at (a5) {$u_4$};
\node [above] at (b2) {$u_5$};
\node [below] at (a1) {$u_6$};
\node [below] at (a6) {$u_7$};
\node [above] at (u8) {$u_8$};
\node [above] at (u10) {$u_{10}$};
\node [above] at (upost) {$u_{n-3}$};
\node [above] at (u9) {$u_9$};
\node [above] at (u11) {$u_{11}$};
\node [above] at (uost) {$u_{n-2}$};
\wierz{un}{$(uost)+(0,1)$};
\node [above] at (un) {$u_{n-1}$};
\krn{uost}{un};
 \end{tikzpicture}
\end{center}

Builder selects $u_{n-1}u_{n-3}$. 
Suppose Painter colours it red. Then the red graph induced on the vertex set $\{u_0,u_1,\ldots, u_{n-4}\}$
is an $(s,s)$-butterfly with centers $u_0$, $u_1$ and Lemma \ref{motyle}\ref{bb}  implies that
Builder can force a blue path $P_{n-3}$ with ends $u_0$, $u_1$, on the set $\{u_0,u_1,\ldots, u_{n-4}\}$.
Then he selects the edges $u_1u_{n-1}$ and  $u_{n-2}u_{n-3}$; both edges must be coloured blue 
so that a red $C_4$ is avoided. The blue path $u_1u_{n-1}u_{n-2}u_{n-3}$ extends the blue path $P_{n-3}$ 
to a blue path $P_n$.
The number of rounds made so far is $e(H)+1+e(P_{n-3})+2=2n-2$ and exactly $n-1$ edges are blue
so Builder wins $\RR(C_4,P_{n})$.

Suppose now that the edge $u_{n-1}u_{n-3}$ selected in round $n$ is blue. 
Then the red graph induced on the vertex set $\{u_0,u_1,\ldots, u_{n-2}\}$
is an $(s,s)$-butterfly with centers $u_0$, $u_1$, $s\ge 2$, and Builder 
can force a blue path on the set $V(H)$ (with ends $u_0,u_1$) in a way similar to the one in the proof of 
Lemma \ref{motyle}\ref{bb}. Builder can force a blue path from $u_0$ to $u_1$ through all vertices $\{u_2,u_3,\ldots, u_{n-2}\}$ such that $u_{n-3}$ and $u_{n-2}$ are adjacent. He selects all edges from this path except $u_{n-2}u_{n-3}$ since these two vertices have been already connected with the blue path 
$u_{n-2}u_{n-1}u_{n-3}$. Thus a blue path $P_n$ is created after  
$2n-3$ rounds of $\RR(C_4,P_{n})$.

\subsubsection{First blue edge in Stage 2}

We consider a few cases, depending on the time Painter colours an edge blue for the first time. 

\paragraph{}
First blue edge in round $n-j$, where either $j\in\{2,3\}$ or $n$ is even and $j\in\{1,4\}$.

\medskip\noindent
Then for the next $j-1\le 3$ rounds (if any) Builder selects edges as in Stage 2, starting from an edge incident to $u_1$, 
no matter how Painter responds. Thus the wings of the coloured butterfly remain almost equal. 
By case analysis we infer that after these $j-1$ rounds the obtained coloured graph $H$
is up to symmetry an $(s',s)$-butterfly on $n$ vertices and it fulfills the assumption of Lemma \ref{malo} (let us recall that $n\ge 14$). 
Further in the game $\RR(C_4,P_n)$ Builder plays according to a winning strategy of Builder in $\RR(C_4,P_{v(H)},H)$
and thereby he wins $\RR(C_4,P_n)$.

\paragraph{}
First blue edge in round $t$ for $10\le t\le n-5$ or in an odd round $t=n-4$.

\medskip\noindent
Let $H$ be the graph induced by $t$ coloured edges.
Suppose that $t$ is even. Then $10\le t\le n-5$, the $t-1$ red edges of $H$ induce a red $(s,s)$-butterfly $B$ with 
centers $u_0,u_1$, $s\ge 2$ and there is one blue edge, say $u_0x$,  selected in round $t$.  
Based on Lemma \ref{motyle}\ref{br}, Builder forces a blue path $P_{t-1}$ on the vertex set $V(B)\setm\{u_1\}$,
with ends $u_0$, $y$ for a vertex $y$ from the $u_1$-wing. 
The blue edge $u_0x$ extends $P_{t-1}$ to a blue path $P_{t}$. Let $A$ be the set of internal vertices of $P_{t}$, i.e.
$A=V(H)\setminus\{y,x,u_0\}$. Then $|A|=t-2$, $N_H(A)$ consists of $t-1$ blue edges of the path $P_{t}$ and 
$|E(B)\setm\{u_1y\}|=t-2$ red edges, and the ends of $P_{t}$ are not adjacent. Thus  
$|N_H(A)|+\delta_H(u_1,x)=2t-3=2|A|+1$.
The graph $H'$ obtained from $H$ by contraction  $P_{t}$ to an edge $xy$ is a $br$-path. 
Note that $n-|A|=n-t+2\ge 7$. If $n-|A|= 7$, we apply Lemma \ref{warpocz}; 
if  $n-|A|\ge 8$, we apply the inductive hypothesis; 
in both cases we infer that Builder has a winning strategy in $\RR(C_4,P_{n-|A|},H')$. Then, by Lemma \ref{sciag}, 
we conclude that  Builder has a winning strategy in $\RR(C_4,P_n,H)$.  

Suppose now that $t$ is odd. Then $11\le t\le n-4$ and the $t-1$ red edges of $H$ induce a red $(s+1,s)$-butterfly $B$ with 
centers $u_0,u_1$, $s\ge 2$, and there is one blue edge, say $u_1x$,  selected in round $t$.  Let $B'$ be an $(s,s)$-butterfly
obtained from $B$ by deleting an edge $u_0z$ of the $u_0$-wing.  
Based on Lemma \ref{motyle}\ref{br}, Builder forces a blue path $P_{t-2}$ on the vertex set $V(B')\setm\{u_0\}$, 
with ends $u_1$, $y$ for a vertex $y$ from the $u_0$-wing. 
The blue edge $u_1x$ extends $P_{t-2}$ to a blue path $P_{t-1}$. Let $A$ be the set of internal vertices of $P_{t-1}$, i.e.
$A=V(H)\setminus\{u_0,x,y,z\}$. Then $|A|=t-3$, $N_H(A)$ consists of $t-2$ blue edges of the path $P_{t-1}$ and 
$|E(B')\setm\{u_0y\}|=t-3$ red edges, and the ends of $P_{t-1}$ are not adjacent. Thus  
$|N_H(A)|+\delta_H(u_0,x)=2t-5=2|A|+1$.
The graph $H'$ obtained from $H$ by contraction  $P_{t-1}$ to an edge $xy$ is a $brr$-path $x y u_0 z$. 
Note that $n-|A|=n-t+3\ge 7$. If $n-|A|= 7$, we apply Lemma \ref{warpocz}; 
if  $n-|A|\ge 8$, we apply the inductive hypothesis; 
in both cases we infer that Builder has a winning strategy in $\RR(C_4,P_{n-|A|},H')$. 
In view of Lemma \ref{sciag}, Builder has a winning strategy in $\RR(C_4,P_n,H)$.  

\paragraph{}
First blue edge in round $t\in\{8,9\}$.

\medskip\noindent
The analysis is very similar to the argument in the previous case, but we use Lemma \ref{motyle}\ref{br}
instead of Lemma \ref{motyle}\ref{bb} and the result of contraction is a smaller path. 
For the sake of completeness, we present the details.  

Let $H$ be the graph induced by $t$ coloured edges.

Suppose that $t=8$. Then $t\le n-6$ since $n\ge 14$. The $t-1$ red edges of $H$ induce a red $(1,1)$-butterfly with 
centers $u_0,u_1$ and there is one blue edge, say $u_0x$,  selected in round $t$.  
Based on Lemma \ref{motyle}\ref{bb}, Builder forces a blue path $P_t$ with the vertex set of the $(1,1)$-butterfly, with ends $u_0$, $u_1$. 
The blue edge $u_0x$ extends $P_t$ to a blue path $P_{t+1}$. Let $A$ be the set of internal vertices of $P_{t+1}$, i.e.
$A=V(H)\setminus\{u_1,x\}$. Then $|A|=t-1$, $N_H(A)$ consists of $t$ blue edges of the path $P_{t+1}$ and $t-1$ 
red edges of $B$, and the ends of $P_{t+1}$ are not adjacent. Thus  
$|N_H(A)|+\delta_H(u_1,x)=2t-1=2|A|+1$.
The graph $H'$ obtained from $H$ by contraction  $P_{t+1}$ to an edge $u_1x$ is clearly a $b$-path. 
Note that $n-|A|=n-t+1\ge 7$. If $n-|A|= 7$, we apply Lemma \ref{warpocz}; if  $n-|A|\ge 8$, we apply the inductive hypothesis; in
both cases we infer that Builder has a winning strategy in $\RR(C_4,P_{n-|A|},H')$. Then, it follows from Lemma \ref{sciag}, 
that  Builder has a winning strategy in $\RR(C_4,P_n,H)$.  

Suppose now that $t=9$. Then $t\le n-5$. The $t-1$ red edges of $H$ induce a red $(s+1,s)$-butterfly $B$ with 
centers $(u_0,u_1)$ and there is one blue edge, say $u_1x$,  selected in round $t$.  Let $B'$ be an $(s,s)$-butterfly
obtained from $B$ by deleting an edge $u_0y$ of the $u_0$-wing.  
Based on Lemma \ref{motyle}\ref{bb}, Builder forces a blue path $P_{t-1}$ with the vertex set of the $B'$, 
with ends $u_0$, $u_1$. 
The blue edge $u_1x$ extends $P_{t-1}$ to a blue path $P_{t}$. Let $A$ be the set of internal vertices of $P_{t}$, i.e.
$A=V(H)\setminus\{u_0,x\}$. Then $|A|=t-2$, $N_H(A)$ consists of $t-1$ blue edges of the path $P_{t}$ and 
$t-2$ red edges of $B'$, and the ends of $P_{t}$ are not adjacent. Thus  
$|N_H(A)|+\delta_H(u_0,x)=2t-3=2|A|+1$.
The graph $H'$ obtained from $H$ by contraction  $P_{t}$ to an edge $u_0x$ is a $br$-path $y u_0 x$. 
We have $n-|A|=n-t+2\ge 7$. If $n-|A|= 7$, we apply Lemma \ref{warpocz}; if  $n-|A|\ge 8$, we apply the inductive hypothesis; in
both cases Builder has a winning strategy in $\RR(C_4,P_{n-|A|},H')$. Then by Lemma \ref{sciag}  
we finish the argument.

So far we have described a strategy of Builder in the game $\RR(C_4, P_n)$ provided the first blue edge appears 
in Stage 2, Stage 3, or later. In the next subsection, we analyse the game when Painter colours an edge blue sooner. 

\subsubsection{Blue edge in Stage 1}\label{stage1}

\paragraph{}\label{nieb7}
First blue edge in round 7.

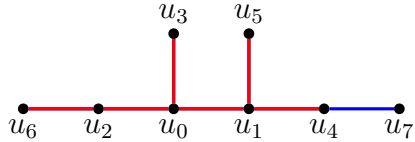
\begin{center}
\begin{tikzpicture}
\jezobc{7};
 \end{tikzpicture}
\end{center}

Builder forces a blue path $u_4u_2u_5u_3u_6u_1$, which  together with the blue edge $u_7u_4$ forms a blue path
$P_7$ with ends $u_7,u_4$. Let $H$ be the graph induced by all coloured edges at that moment. 
One can verify that for $A=\{u_4,u_2,u_5,u_3,u_6\}$ we have
$|N_H(A)|+\delta_H(u_7,u_1)=11=2|A|+1$.
The graph $H'$ obtained from $H$ by contraction  $P$ to an edge $u_7u_1$ is a $br$-path $u_7u_1u_0$. 
In short, we present the argument in the following picture.

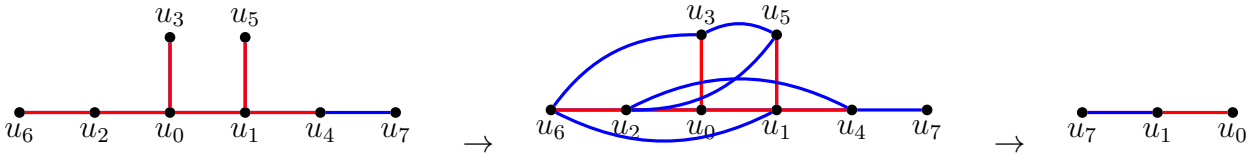
\begin{tikzpicture}
\jezobc{7};
 \end{tikzpicture}
\quad$\to$\quad
\begin{tikzpicture}
\jezobc{7};
\sciezkan{u4}{u2,u5,u3,u6,u1};
 \end{tikzpicture}
\quad$\to$\quad
\begin{tikzpicture}
\sciez{3};
\krn{a1}{a2};
\node [below] at (a1) {$u_7$};
\node [below] at (a2) {$u_1$};
\node [below] at (a3) {$u_0$};
 \end{tikzpicture}

Note that $n-|A|= n-5\ge 9$. By the inductive hypothesis we infer that Builder has a winning strategy in $\RR(C_4,P_{n-|A|},H')$. It follows from Lemma \ref{sciag} that  Builder has a winning strategy in $\RR(C_4,P_n,H)$.

\paragraph{}\label{nieb6}
First blue edge in round 6.

\medskip\noindent
The analysis is similar to these in case \ref{nieb7} so we present the pictures only and updated calculations.

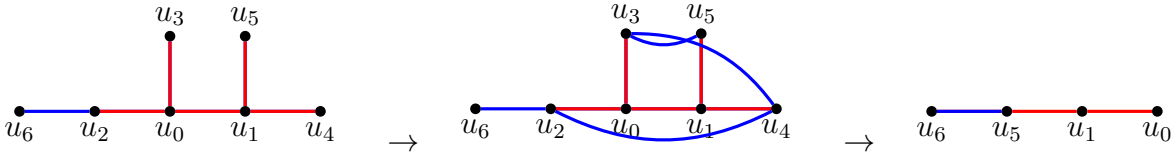
\begin{tikzpicture}
\jezobc{6};
 \end{tikzpicture}
\quad$\to$\quad
\begin{tikzpicture}
\jezobc{6};
\sciezkan{u2}{u4,u3,u5};
 \end{tikzpicture}
\quad$\to$\quad
\begin{tikzpicture}
\sciez{4};
\krn{a1}{a2};
\node [below] at (a1) {$u_6$};
\node [below] at (a2) {$u_5$};
\node [below] at (a3) {$u_1$};
\node [below] at (a4) {$u_0$};
 \end{tikzpicture}

For $A=\{u_2,u_4,u_3\}$ we have
$|N_H(A)|+\delta_H(u_6,u_5)=7=2|A|+1$
and $n-|A|= n-3\ge 11$.

\paragraph{}\label{nieb5}
First blue edge in round 5.

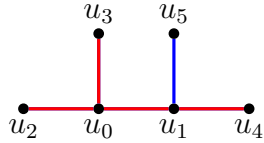
\begin{center}
\begin{tikzpicture}
\jezobc{5};
 \end{tikzpicture}
\end{center}

Here (and in the following cases) the argument is analogous to that in the previous cases. 
However, before some contraction, we consider a few rounds more. 
Builder's move in round 6 and all possible Painter's choices are
shown in the table below.  
We assume that after Builder forces some blue edges and the obtained coloured graph  
is denoted by $H$, shown in the table. Only then we make a path contraction.  
The column last but one shows the vertices of the contracted path and the graph $H'$ obtained from $H$ by 
contraction of the path. The last column contains the number  
$$m=|N_H(A)|+\delta_H(x,y),$$ 
where $A$ is the set of internal vertices of the contracted path and $x,y$ are its ends. 
This notation is also used in the remaining cases.

\medskip
\begin{tabular}{|c|c|c|c|}
\hline
after round 6& $H$ & contracted  path and $H'$ & $m$\\
\hline
\begin{tikzpicture}
\jezobc{5};
\wierz{u6}{$(u2)-(1,0)$};
\node [below] at (u6) {$u_6$};
\kr{u2}{u6};
 \end{tikzpicture}
&
\begin{tikzpicture}
\jezobc{5};
\wierz{u6}{$(u2)-(1,0)$};
\node [below] at (u6) {$u_6$};
\kr{u2}{u6};
\sciezkan{u1}{u6,u3,u4,u2};
 \end{tikzpicture}
&
\begin{tabular}{c}
$u_5u_1u_6u_3u_4u_2$\\[10pt]
\begin{tikzpicture}
\sciez{3};
\krn{a1}{a2};
\node [below] at (a1) {$u_5$};
\node [below] at (a2) {$u_2$};
\node [below] at (a3) {$u_0$};
\end{tikzpicture}\\
\end{tabular}
&9\\
\hline
\begin{tikzpicture}
\jezobc{5};
\wierz{u6}{$(u2)-(1,0)$};
\node [below] at (u6) {$u_6$};
\krn{u2}{u6};
 \end{tikzpicture}
&
\begin{tikzpicture}
\jezobc{5};
\wierz{u6}{$(u2)-(1,0)$};
\node [below] at (u6) {$u_6$};
\krn{u2}{u6};
\sciezkan{u2}{u4,u3};
 \end{tikzpicture}
&
\begin{tabular}{c}
$u_6u_2u_4u_3$\\[10pt]
\begin{tikzpicture}
\sciez{5};
\krn{a1}{a2};
\krn{a4}{a5};
\node [below] at (a1) {$u_6$};
\node [below] at (a2) {$u_3$};
\node [below] at (a3) {$u_0$};
\node [below] at (a4) {$u_1$};
\node [below] at (a5) {$u_5$};
\end{tikzpicture}\\
\end{tabular}
&5\\
\hline
\end{tabular}

\medskip
One can easily verify that $m\le 2|A|+1$ and $n-|A|\ge 8$. Therefore by the inductive hypothesis Builder has a winning strategy in $\RR(C_4,P_{n-|A|},H')$. Then based on Lemma \ref{sciag} we conclude that  Builder has a winning strategy in $\RR(C_4,P_n,H)$. 

\paragraph{}\label{nieb4}
First blue edge in round 4.

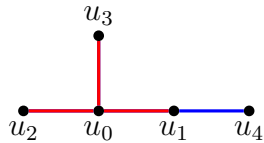
\begin{center}
\begin{tikzpicture}
\jezobc{4};
 \end{tikzpicture}
\end{center}

Builder selects an edge $u_1u_5$, where $u_5$ is any free vertex. 
If Painter colours it red then we obtain the coloured graph isomorphic to the graph after round 5, considered in case \ref{nieb5}.
Thus we can assume that Painter colours  $u_1u_5$ blue and the resulting graph is 

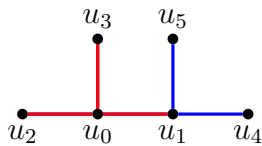
\begin{center}
\begin{tikzpicture}
\jezobc{4};
\wierz{u5}{$(u1)+(0,1)$};
\node [above] at (u5) {$u_5$};
\krn{u1}{u5};
 \end{tikzpicture}
\end{center}

Then Builder selects $u_3u_5$. If Painter colours it blue, we obtain the coloured graph $H$ on 6 vertices, with the blue path $u_4u_1u_5u_3$
such that $N_H(\{u_1,u_5\})+\delta_H(u_4,u_3)=4$. After the contraction of this path, we get the graph $H'$ which is a $brr$-path.

 If Painter colours $u_3u_5$ red, Builder in round 6 forces a blue edge $u_5u_2$, and the resulting coloured graph $H$ on 6 vertices
has the blue path $u_4u_1u_5u_2$
such that $N_H(\{u_1,u_5\})+\delta_H(u_4,u_2)=5$. After the contraction of this path, we get the graph $H'$ which is a $brr$-path.

In both cases we have $m\le 2|A|+1$ and $n-|A|\ge 8$ so one can argue as before that Builder wins the game.

\paragraph{}\label{nieb3}
First blue edge in round 3.

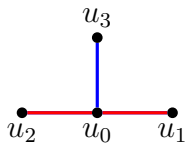
\begin{center}
\begin{tikzpicture}
\jezobc{3};
 \end{tikzpicture}
\end{center}

Builder selects two edges $u_2u_4$ and $u_2u_5$, where $u_4,u_5$ are any free vertices. Then Painter colours  
$u_2u_4$ and $u_2u_5$. If both edges are red, we receive a coloured graph isomorphic to the graph after round 5, considered in case \ref{nieb5}. Therefore further we assume that at least one edge $u_2u_4$, $u_2,u_5$  is blue. 
Consider two possibilities.

Suppose first that exactly one of the two edges is blue, say $u_2u_5$. Then Builder selects $u_2u_6$ with any free vertex $u_6$. 
The further play, split into cases depending on Painter's choice in rounds 6, 7 and 8, is presented in the table below. 
Let us recall that the graph $H$ shown in the table may contain additional blue edges, forced by Builder after 
the rounds presented step by step. The notation {\em n/a} in the table means that we do not present Builder's move 
in this round.

\medskip
\scalebox{0.74}{
 \begin{tabular}{|c|c|c|c|c|c|}
\hline
after round 6& after round 7& after round 8 & $H$ & \makecell{contracted  path\\ and $H'$} & $m$\\
\hline
\multirow{2}{*}{
\begin{tikzpicture}
\jezobc{3};
\wierz{u4}{$(u2)-(1,0)$};
\node [below] at (u4) {$u_4$};
\kr{u2}{u4};
\wierz{u5}{$(u2)+(0,1)$};
\node [above] at (u5) {$u_5$};
\krn{u2}{u5};
\wierz{u6}{$(u2)-(0,1)$};
\node [left] at (u6) {$u_6$};
\krn{u2}{u6};
 \end{tikzpicture}
}
&
\begin{tikzpicture}
\jezobc{3};
\wierz{u4}{$(u2)-(1,0)$};
\node [below] at (u4) {$u_4$};
\kr{u2}{u4};
\wierz{u5}{$(u2)+(0,1)$};
\node [above] at (u5) {$u_5$};
\krn{u2}{u5};
\wierz{u6}{$(u2)-(0,1)$};
\node [left] at (u6) {$u_6$};
\krn{u2}{u6};
\krn{u0}{u5};
 \end{tikzpicture}
&
\makecell[b]{n/a\\ \ }
&
\begin{tikzpicture}
\jezobc{3};
\wierz{u4}{$(u2)-(1,0)$};
\node [below] at (u4) {$u_4$};
\kr{u2}{u4};
\wierz{u5}{$(u2)+(0,1)$};
\node [above] at (u5) {$u_5$};
\krn{u2}{u5};
\wierz{u6}{$(u2)-(0,1)$};
\node [left] at (u6) {$u_6$};
\krn{u2}{u6};
\krn{u0}{u5};
 \end{tikzpicture}
&
\makecell[b]{
$u_3u_0u_5u_2u_6$\\
\\
\begin{tikzpicture}
\sciez{2};
\krn{a1}{a2};
\node [below] at (a1) {$u_3$};
\node [below] at (a2) {$u_6$};
 \end{tikzpicture}
}
&7\\
\cline{2-6}
  &
\begin{tikzpicture}
\jezobc{3};
\wierz{u4}{$(u2)-(1,0)$};
\node [below] at (u4) {$u_4$};
\kr{u2}{u4};
\wierz{u5}{$(u2)+(0,1)$};
\node [above] at (u5) {$u_5$};
\krn{u2}{u5};
\wierz{u6}{$(u2)-(0,1)$};
\node [left] at (u6) {$u_6$};
\krn{u2}{u6};
\kr{u0}{u5};
 \end{tikzpicture}
&
\makecell[b]{n/a\\ \ }
&
\begin{tikzpicture}
\jezobc{3};
\wierz{u4}{$(u2)-(1,0)$};
\node [below] at (u4) {$u_4$};
\kr{u2}{u4};
\wierz{u5}{$(u2)+(0,1)$};
\node [above] at (u5) {$u_5$};
\krn{u2}{u5};
\wierz{u6}{$(u2)-(0,1)$};
\node [left] at (u6) {$u_6$};
\krn{u2}{u6};
\kr{u0}{u5};
\sciezkan{u5}{u4,u1};
 \end{tikzpicture}
&
\makecell[b]{
$u_6u_2u_5u_4u_1$\\
\\
\begin{tikzpicture}
\sciez{4};
\krn{a1}{a2};
\krn{a3}{a4};
\node [below] at (a1) {$u_6$};
\node [below] at (a2) {$u_1$};
\node [below] at (a3) {$u_0$};
\node [below] at (a4) {$u_3$};
 \end{tikzpicture}
}
&7\\
\hlineB{3}
\multirow{3}{*}{
\begin{tikzpicture}
\jezobc{3};
\wierz{u4}{$(u2)-(1,0)$};
\node [below] at (u4) {$u_4$};
\kr{u2}{u4};
\wierz{u5}{$(u2)+(0,1)$};
\node [above] at (u5) {$u_5$};
\krn{u2}{u5};
\wierz{u6}{$(u2)-(0,1)$};
\node [left] at (u6) {$u_6$};
\kr{u2}{u6};
 \end{tikzpicture}
}
&
\begin{tikzpicture}
\jezobc{3};
\wierz{u4}{$(u2)-(1,0)$};
\node [below] at (u4) {$u_4$};
\kr{u2}{u4};
\wierz{u5}{$(u2)+(0,1)$};
\node [above] at (u5) {$u_5$};
\krn{u2}{u5};
\wierz{u6}{$(u2)-(0,1)$};
\node [left] at (u6) {$u_6$};
\kr{u2}{u6};
\wierz{u7}{$(u6)-(0,1)$};
\node [left] at (u7) {$u_7$};
\kr{u6}{u7};
 \end{tikzpicture}
&
\makecell[b]{n/a\\ \ }
&
\begin{tikzpicture}
\jezobc{3};
\wierz{u4}{$(u2)-(1,0)$};
\node [below] at (u4) {$u_4$};
\kr{u2}{u4};
\wierz{u5}{$(u2)+(0,1)$};
\node [above] at (u5) {$u_5$};
\krn{u2}{u5};
\wierz{u6}{$(u2)-(0,1)$};
\node [left] at (u6) {$u_6$};
\kr{u2}{u6};
\wierz{u7}{$(u6)-(0,1)$};
\node [left] at (u7) {$u_7$};
\kr{u6}{u7};
\sciezkan{u6}{u1,u4,u7,u0};
 \end{tikzpicture}
&
\makecell[b]{
$u_3u_0u_7u_4u_1u_6$\\
\\
\begin{tikzpicture}
\sciez{4};
\krn{a1}{a2};
\krn{a3}{a4};
\node [below] at (a1) {$u_3$};
\node [below] at (a2) {$u_6$};
\node [below] at (a3) {$u_2$};
\node [below] at (a4) {$u_5$};
 \end{tikzpicture}
}
&9\\
\clineB{2-6}{3}
     &
\multirow{2}{*}{
\begin{tikzpicture}
\jezobc{3};
\wierz{u4}{$(u2)-(1,0)$};
\node [below] at (u4) {$u_4$};
\kr{u2}{u4};
\wierz{u5}{$(u2)+(0,1)$};
\node [above] at (u5) {$u_5$};
\krn{u2}{u5};
\wierz{u6}{$(u2)-(0,1)$};
\node [left] at (u6) {$u_6$};
\kr{u2}{u6};
\wierz{u7}{$(u6)-(0,1)$};
\node [left] at (u7) {$u_7$};
\krn{u6}{u7};
 \end{tikzpicture}
}
&
\begin{tikzpicture}
\jezobc{3};
\wierz{u4}{$(u2)-(1,0)$};
\node [below] at (u4) {$u_4$};
\kr{u2}{u4};
\wierz{u5}{$(u2)+(0,1)$};
\node [above] at (u5) {$u_5$};
\krn{u2}{u5};
\wierz{u6}{$(u2)-(0,1)$};
\node [left] at (u6) {$u_6$};
\kr{u2}{u6};
\wierz{u7}{$(u6)-(0,1)$};
\node [left] at (u7) {$u_7$};
\krn{u6}{u7};
\kr{u5}{u0};
 \end{tikzpicture}
&
\begin{tikzpicture}
\jezobc{3};
\wierz{u4}{$(u2)-(1,0)$};
\node [below] at (u4) {$u_4$};
\kr{u2}{u4};
\wierz{u5}{$(u2)+(0,1)$};
\node [above] at (u5) {$u_5$};
\krn{u2}{u5};
\wierz{u6}{$(u2)-(0,1)$};
\node [left] at (u6) {$u_6$};
\kr{u2}{u6};
\wierz{u7}{$(u6)-(0,1)$};
\node [left] at (u7) {$u_7$};
\krn{u6}{u7};
\kr{u5}{u0};
\sciezkan{u6}{u1,u4,u5};
 \end{tikzpicture}
&
\makecell[b]{
$u_7u_6u_1u_4u_5u_2$\\
\\
\begin{tikzpicture}
\sciez{4};
\krn{a1}{a2};
\krn{a3}{a4};
\node [below] at (a1) {$u_7$};
\node [below] at (a2) {$u_2$};
\node [below] at (a3) {$u_0$};
\node [below] at (a4) {$u_3$};
 \end{tikzpicture}
}
&9\\
\cline{3-6}
  &
  &
\begin{tikzpicture}
\jezobc{3};
\wierz{u4}{$(u2)-(1,0)$};
\node [below] at (u4) {$u_4$};
\kr{u2}{u4};
\wierz{u5}{$(u2)+(0,1)$};
\node [above] at (u5) {$u_5$};
\krn{u2}{u5};
\wierz{u6}{$(u2)-(0,1)$};
\node [left] at (u6) {$u_6$};
\kr{u2}{u6};
\wierz{u7}{$(u6)-(0,1)$};
\node [left] at (u7) {$u_7$};
\krn{u6}{u7};
\krn{u5}{u0};
\end{tikzpicture}
&
\begin{tikzpicture}
\jezobc{3};
\wierz{u4}{$(u2)-(1,0)$};
\node [below] at (u4) {$u_4$};
\kr{u2}{u4};
\wierz{u5}{$(u2)+(0,1)$};
\node [above] at (u5) {$u_5$};
\krn{u2}{u5};
\wierz{u6}{$(u2)-(0,1)$};
\node [left] at (u6) {$u_6$};
\kr{u2}{u6};
\wierz{u7}{$(u6)-(0,1)$};
\node [left] at (u7) {$u_7$};
\krn{u6}{u7};
\krn{u5}{u0};
\sciezkan{u6}{u1,u4};
\end{tikzpicture}
&
\makecell[b]{
$u_7u_6u_1u_4$\\ 
then\\ 
$u_3u_0u_5u_2$\\
\\
\begin{tikzpicture}
\sciez{4};
\krn{a1}{a2};
\krn{a3}{a4};
\node [below] at (a1) {$u_7$};
\node [below] at (a2) {$u_4$};
\node [below] at (a3) {$u_2$};
\node [below] at (a4) {$u_3$};
 \end{tikzpicture}
}
&
\makecell[b]{5\\ then\\ 4}\\

\hline
\end{tabular}
} 

\medskip
Let us comment on the last row, i.e.~when Painter colours blue edges $u_6u_7$ and $u_0u_5$ selected in rounds 7 and 8, respectively.
After round 8 Builder forces the blue path $u_6u_1u_4$. The blue path $u_7u_6u_1u_4$ has non-adjacent ends 
and 5 coloured edges incident to its internal vertices. 
So $N_H(\{u_6,u_1\})+\delta_H(u_7,u_4)=5\le 2|\{u_6,u_1\}|+1$.
We contract the path $u_7u_6u_1u_4$ and obtain a new coloured graph, let us denote it by $H''$, in which there is 
the blue path $u_3u_0u_5u_2$ with non-adjacent ends and 4 coloured edges
incident to its internal vertices. Thus $N_{H''}(\{u_0,u_5\})+\delta_H(u_3,u_2)=4\le 2|\{u_0,u_5\}|+1$.
We contract the path $u_3u_0u_5u_2$ and the graph $H'$ in the table is the result. 
By the inductive hypothesis, Builder has a winning strategy in $\RR(C_4,P_{v(H')},H')$. 
Then Lemma \ref{sciag} implies that Builder has a winning strategy in $\RR(C_4,P_{v(H'')},H'')$.
Finally, we apply Lemma \ref{sciag} again and conclude that Builder wins $\RR(C_4,P_{v(H)},H)$.
The analysis of other cases presented in the above table is simpler and we omit it.

Suppose now that both edges $u_2u_4$ and $u_2u_5$ selected by Builder in rounds 4 and 5 are coloured blue. 
The further moves of Builder and cases depending on Painter's choice in rounds 6, 7 are presented in the table below.

 \medskip
\scalebox{0.9}{
 \begin{tabular}{|c|c|c|c|c|}
\hline
after round 6& after round 7&  $H$ & \makecell{contracted  path\\ and $H'$} & $m$\\
\hline
\multirow{2}{*}{
\begin{tikzpicture}
\jezobc{3};
\wierz{u4}{$(u2)-(1,0)$};
\node [below] at (u4) {$u_4$};
\krn{u2}{u4};
\wierz{u5}{$(u2)+(0,1)$};
\node [above] at (u5) {$u_5$};
\krn{u2}{u5};
\kr{u1}{u3};
 \end{tikzpicture}
}
&
\begin{tikzpicture}
\jezobc{3};
\wierz{u4}{$(u2)-(1,0)$};
\node [below] at (u4) {$u_4$};
\krn{u2}{u4};
\wierz{u5}{$(u2)+(0,1)$};
\node [above] at (u5) {$u_5$};
\krn{u2}{u5};
\kr{u1}{u3};
\kr{u5}{u3};
 \end{tikzpicture}
&
\begin{tikzpicture}
\jezobc{3};
\wierz{u4}{$(u2)-(1,0)$};
\node [below] at (u4) {$u_4$};
\krn{u2}{u4};
\wierz{u5}{$(u2)+(0,1)$};
\node [above] at (u5) {$u_5$};
\krn{u2}{u5};
\kr{u1}{u3};
\kr{u5}{u3};
\krn{u0}{u5};
 \end{tikzpicture}
&
\makecell[b]{
$u_4u_2u_5u_0u_3$\\ 
\\
\begin{tikzpicture}
\sciez{3};
\krn{a1}{a2};
\node [below] at (a1) {$u_4$};
\node [below] at (a2) {$u_3$};
\node [below] at (a3) {$u_1$};
 \end{tikzpicture}
}
&
7\\

\cline{2-5}
  &
\begin{tikzpicture}
\jezobc{3};
\wierz{u4}{$(u2)-(1,0)$};
\node [below] at (u4) {$u_4$};
\krn{u2}{u4};
\wierz{u5}{$(u2)+(0,1)$};
\node [above] at (u5) {$u_5$};
\krn{u2}{u5};
\kr{u1}{u3};
\krn{u3}{u5};
 \end{tikzpicture}
&
\begin{tikzpicture}
\jezobc{3};
\wierz{u4}{$(u2)-(1,0)$};
\node [below] at (u4) {$u_4$};
\krn{u2}{u4};
\wierz{u5}{$(u2)+(0,1)$};
\node [above] at (u5) {$u_5$};
\krn{u2}{u5};
\kr{u1}{u3};
\krn{u3}{u5};
 \end{tikzpicture}
&
\makecell[b]{
$u_4u_2u_5u_3u_0$\\
\\
\begin{tikzpicture}
\sciez{3};
\krn{a1}{a2};
\node [below] at (a1) {$u_4$};
\node [below] at (a2) {$u_0$};
\node [below] at (a3) {$u_1$};
 \end{tikzpicture}
}
&6\\
\hlineB{3}
\multirow{2}{*}{
\begin{tikzpicture}
\jezobc{3};
\wierz{u4}{$(u2)-(1,0)$};
\node [below] at (u4) {$u_4$};
\krn{u2}{u4};
\wierz{u5}{$(u2)+(0,1)$};
\node [above] at (u5) {$u_5$};
\krn{u2}{u5};
\krn{u1}{u3};
 \end{tikzpicture}
}
&
\begin{tikzpicture}
\jezobc{3};
\wierz{u4}{$(u2)-(1,0)$};
\node [below] at (u4) {$u_4$};
\krn{u2}{u4};
\wierz{u5}{$(u2)+(0,1)$};
\node [above] at (u5) {$u_5$};
\krn{u2}{u5};
\kr{u0}{u5};
\krn{u1}{u3};
 \end{tikzpicture}
&
\begin{tikzpicture}
\jezobc{3};
\wierz{u4}{$(u2)-(1,0)$};
\node [below] at (u4) {$u_4$};
\krn{u2}{u4};
\wierz{u5}{$(u2)+(0,1)$};
\node [above] at (u5) {$u_5$};
\krn{u2}{u5};
\kr{u0}{u5};
\krn{u1}{u3};
 \end{tikzpicture}
&
\makecell[b]{
$u_4u_2u_5$\\ 
then\\ 
$u_0u_3u_1$\\
\\
\begin{tikzpicture}
\sciez{4};
\krn{a1}{a2};
\krn{a3}{a4};
\node [below] at (a1) {$u_4$};
\node [below] at (a2) {$u_5$};
\node [below] at (a3) {$u_0$};
\node [below] at (a4) {$u_1$};
 \end{tikzpicture}
}
&
\makecell[b]{3\\ then\\ 3}\\

\cline{2-5}
  &
\begin{tikzpicture}
\jezobc{3};
\wierz{u4}{$(u2)-(1,0)$};
\node [below] at (u4) {$u_4$};
\krn{u2}{u4};
\wierz{u5}{$(u2)+(0,1)$};
\node [above] at (u5) {$u_5$};
\krn{u2}{u5};
\krn{u0}{u5};
\krn{u1}{u3};
 \end{tikzpicture}
&
\begin{tikzpicture}
\jezobc{3};
\wierz{u4}{$(u2)-(1,0)$};
\node [below] at (u4) {$u_4$};
\krn{u2}{u4};
\wierz{u5}{$(u2)+(0,1)$};
\node [above] at (u5) {$u_5$};
\krn{u2}{u5};
\krn{u0}{u5};
\krn{u1}{u3};
 \end{tikzpicture}
&
\makecell[b]{
$u_4u_2u_5u_0u_3u_1$\\
\\
\begin{tikzpicture}
\sciez{2};
\krn{a1}{a2};
\node [below] at (a1) {$u_4$};
\node [below] at (a2) {$u_1$};
 \end{tikzpicture}
}
&7\\
\hline
\end{tabular}
} 


\medskip
Further, we argue analogously as before so we skip the details. 

\paragraph{}\label{nieb2}
First blue edge in round 2.

\begin{center}
\begin{tikzpicture}
\jezobc{2};
 \end{tikzpicture}
\end{center}

Then in round 3 Builder selects an edge incident to $u_1$ and any free vertex $u_3$. After the colour choice of Painter 
we have either a $brb$-path or a $brr$-path at the board. 

Suppose that $u_2u_0u_1u_3$ is a $brb$-path. The edges selected by Builder in rounds 4 and 5 and their possible colours 
are presented in the table below. We use notation as in previous cases and skip the argument details since they are 
analogous.

  \medskip
\scalebox{0.9}{
 \begin{tabular}{|c|c|c|c|c|}
\hline
after round 4& after round 5&  $H$ & \makecell{contracted  path\\ and $H'$} & $m$\\
\hline
\multirow{2}{*}{
\begin{tikzpicture}
\jezobc{2};
\wierz{u3}{$(u1)+(1,0)$};
\node [below] at (u3) {$u_3$};
\krn{u1}{u3};
\wierz{u4}{$(u1)+(0,1)$};
\node [above] at (u4) {$u_4$};
\kr{u1}{u4};
 \end{tikzpicture}
}
&
\begin{tikzpicture}
\jezobc{2};
\wierz{u3}{$(u1)+(1,0)$};
\node [below] at (u3) {$u_3$};
\krn{u1}{u3};
\wierz{u4}{$(u1)+(0,1)$};
\node [above] at (u4) {$u_4$};
\kr{u1}{u4};
\kr{u4}{u3};
 \end{tikzpicture}
&
\begin{tikzpicture}
\jezobc{2};
\wierz{u3}{$(u1)+(1,0)$};
\node [below] at (u3) {$u_3$};
\krn{u1}{u3};
\wierz{u4}{$(u1)+(0,1)$};
\node [above] at (u4) {$u_4$};
\kr{u1}{u4};
\kr{u4}{u3};
\sciezkan{u0}{u3};
 \end{tikzpicture}
&
\makecell[b]{
$u_2u_0u_3u_1$\\ 
\\
\begin{tikzpicture}
\sciez{3};
\krn{a1}{a2};
\node [below] at (a1) {$u_2$};
\node [below] at (a2) {$u_1$};
\node [below] at (a3) {$u_4$};
 \end{tikzpicture}
}
&
5\\

\cline{2-5}
  &
\begin{tikzpicture}
\jezobc{2};
\wierz{u3}{$(u1)+(1,0)$};
\node [below] at (u3) {$u_3$};
\krn{u1}{u3};
\wierz{u4}{$(u1)+(0,1)$};
\node [above] at (u4) {$u_4$};
\kr{u1}{u4};
\krn{u4}{u3};
 \end{tikzpicture}
&
\begin{tikzpicture}
\jezobc{2};
\wierz{u3}{$(u1)+(1,0)$};
\node [below] at (u3) {$u_3$};
\krn{u1}{u3};
\wierz{u4}{$(u1)+(0,1)$};
\node [above] at (u4) {$u_4$};
\kr{u1}{u4};
\krn{u4}{u3};
 \end{tikzpicture}
&
\makecell[b]{
$u_1u_3u_4$\\
\\
\begin{tikzpicture}
\sciez{4};
\krn{a1}{a2};
\krn{a3}{a4};
\node [below] at (a1) {$u_2$};
\node [below] at (a2) {$u_0$};
\node [below] at (a3) {$u_1$};
\node [below] at (a4) {$u_4$};
 \end{tikzpicture}
}
&3\\

\hlineB{3}

\multirow{2}{*}{
\begin{tikzpicture}
\jezobc{2};
\wierz{u3}{$(u1)+(1,0)$};
\node [below] at (u3) {$u_3$};
\krn{u1}{u3};
\wierz{u4}{$(u1)+(0,1)$};
\node [above] at (u4) {$u_4$};
\krn{u1}{u4};
 \end{tikzpicture}
}
&
\begin{tikzpicture}
\jezobc{2};
\wierz{u3}{$(u1)+(1,0)$};
\node [below] at (u3) {$u_3$};
\krn{u1}{u3};
\wierz{u4}{$(u1)+(0,1)$};
\node [above] at (u4) {$u_4$};
\krn{u1}{u4};
\kr{u0}{u4};
 \end{tikzpicture}
&
\begin{tikzpicture}
\jezobc{2};
\wierz{u3}{$(u1)+(1,0)$};
\node [below] at (u3) {$u_3$};
\krn{u1}{u3};
\wierz{u4}{$(u1)+(0,1)$};
\node [above] at (u4) {$u_4$};
\krn{u1}{u4};
\kr{u0}{u4};
 \end{tikzpicture}
&
\makecell[b]{
$u_4u_1u_3$\\ 
\\
\begin{tikzpicture}
\sciez{4};
\krn{a1}{a2};
\krn{a3}{a4};
\node [below] at (a1) {$u_2$};
\node [below] at (a2) {$u_0$};
\node [below] at (a3) {$u_4$};
\node [below] at (a4) {$u_3$};
 \end{tikzpicture}
}
&
3\\

\cline{2-5}
  &
\begin{tikzpicture}
\jezobc{2};
\wierz{u3}{$(u1)+(1,0)$};
\node [below] at (u3) {$u_3$};
\krn{u1}{u3};
\wierz{u4}{$(u1)+(0,1)$};
\node [above] at (u4) {$u_4$};
\krn{u1}{u4};
\krn{u0}{u4};
 \end{tikzpicture}
&
\begin{tikzpicture}
\jezobc{2};
\wierz{u3}{$(u1)+(1,0)$};
\node [below] at (u3) {$u_3$};
\krn{u1}{u3};
\wierz{u4}{$(u1)+(0,1)$};
\node [above] at (u4) {$u_4$};
\krn{u1}{u4};
\krn{u0}{u4};
 \end{tikzpicture}
&
\makecell[b]{
$u_2u_0u_4u_1u_3$\\
\\
\begin{tikzpicture}
\sciez{2};
\krn{a1}{a2};
\node [below] at (a1) {$u_2$};
\node [below] at (a2) {$u_3$};
 \end{tikzpicture}
}
&5\\
\hline
\end{tabular}
} 


\medskip
Assume now that $u_2u_0u_1u_3$ is a $brr$-path. Further play is presented below, cases depend on
Painter's choice in rounds 4, 5, 6.

 \medskip
\hskip-25pt
\scalebox{0.7}{
 \begin{tabular}{|c|c|c|c|c|c|}
\hline
after round 4& after round 5& after round 6 & $H$ & \makecell{contracted  path\\ and $H'$} & $m$\\
\hline

\multirow{3}{*}{
\begin{tikzpicture}
\jezobc{2};
\wierz{u3}{$(u1)+(1,0)$};
\node [below] at (u3) {$u_3$};
\kr{u1}{u3};
\wierz{u4}{$(u3)+(1,0)$};
\node [below] at (u4) {$u_4$};
\krn{u3}{u4};
 \end{tikzpicture}
}
&
\begin{tikzpicture}
\jezobc{2};
\wierz{u3}{$(u1)+(1,0)$};
\node [below] at (u3) {$u_3$};
\kr{u1}{u3};
\wierz{u4}{$(u3)+(1,0)$};
\node [below] at (u4) {$u_4$};
\krn{u3}{u4};
\krnwyg{u4}{u1};
 \end{tikzpicture}
&
\makecell[b]{n/a\\ \ }
&
\begin{tikzpicture}
\jezobc{2};
\wierz{u3}{$(u1)+(1,0)$};
\node [below] at (u3) {$u_3$};
\kr{u1}{u3};
\wierz{u4}{$(u3)+(1,0)$};
\node [below] at (u4) {$u_4$};
\krn{u3}{u4};
\krnwyg{u4}{u1};
 \end{tikzpicture}
&
\makecell[b]{
$u_1u_4u_3$\\
\\
\begin{tikzpicture}
\sciez{4};
\krn{a1}{a2};
\krn{a3}{a4};
\node [below] at (a1) {$u_2$};
\node [below] at (a2) {$u_0$};
\node [below] at (a3) {$u_1$};
\node [below] at (a4) {$u_3$};
 \end{tikzpicture}
}
&3\\
\clineB{2-6}{3}
     &
\multirow{2}{*}{
\begin{tikzpicture}
\jezobc{2};
\wierz{u3}{$(u1)+(1,0)$};
\node [below] at (u3) {$u_3$};
\kr{u1}{u3};
\wierz{u4}{$(u3)+(1,0)$};
\node [below] at (u4) {$u_4$};
\krn{u3}{u4};
\krwyg{u4}{u1};
 \end{tikzpicture}
}
&
\begin{tikzpicture}
\jezobc{2};
\wierz{u3}{$(u1)+(1,0)$};
\node [below] at (u3) {$u_3$};
\kr{u1}{u3};
\wierz{u4}{$(u3)+(1,0)$};
\node [below] at (u4) {$u_4$};
\krn{u3}{u4};
\krwyg{u4}{u1};
\wierz{u5}{$(u3)+(0,1)$};
\node [above] at (u5) {$u_5$};
\kr{u3}{u5};
 \end{tikzpicture}
&
\begin{tikzpicture}
\jezobc{2};
\wierz{u3}{$(u1)+(1,0)$};
\node [below] at (u3) {$u_3$};
\kr{u1}{u3};
\wierz{u4}{$(u3)+(1,0)$};
\node [below] at (u4) {$u_4$};
\krn{u3}{u4};
\krwyg{u4}{u1};
\wierz{u5}{$(u3)+(0,1)$};
\node [above] at (u5) {$u_5$};
\kr{u3}{u5};
\sciezkan{u4}{u5,u0};
 \end{tikzpicture}
&
\makecell[b]{
$u_2u_0u_5u_4u_3$\\
\\
\begin{tikzpicture}
\sciez{3};
\krn{a1}{a2};
\node [below] at (a1) {$u_2$};
\node [below] at (a2) {$u_3$};
\node [below] at (a3) {$u_1$};
 \end{tikzpicture}
}
&7\\
\cline{3-6}
  &
  &
\begin{tikzpicture}
\jezobc{2};
\wierz{u3}{$(u1)+(1,0)$};
\node [below] at (u3) {$u_3$};
\kr{u1}{u3};
\wierz{u4}{$(u3)+(1,0)$};
\node [below] at (u4) {$u_4$};
\krn{u3}{u4};
\krwyg{u4}{u1};
\wierz{u5}{$(u3)+(0,1)$};
\node [above] at (u5) {$u_5$};
\krn{u3}{u5};
 \end{tikzpicture}
&
\begin{tikzpicture}
\jezobc{2};
\wierz{u3}{$(u1)+(1,0)$};
\node [below] at (u3) {$u_3$};
\kr{u1}{u3};
\wierz{u4}{$(u3)+(1,0)$};
\node [below] at (u4) {$u_4$};
\krn{u3}{u4};
\krwyg{u4}{u1};
\wierz{u5}{$(u3)+(0,1)$};
\node [above] at (u5) {$u_5$};
\krn{u3}{u5};
 \end{tikzpicture}
&
\makecell[b]{
$u_4u_3u_5$\\ 
\\
\begin{tikzpicture}
\sciez{5};
\krn{a1}{a2};
\krn{a4}{a5};
\node [below] at (a1) {$u_2$};
\node [below] at (a2) {$u_0$};
\node [below] at (a3) {$u_1$};
\node [below] at (a4) {$u_4$};
\node [below] at (a5) {$u_5$};
 \end{tikzpicture}
}
&
3\\

\hlineB{3}

\begin{tikzpicture}
\jezobc{2};
\wierz{u3}{$(u1)+(1,0)$};
\node [below] at (u3) {$u_3$};
\kr{u1}{u3};
\wierz{u4}{$(u3)+(1,0)$};
\node [below] at (u4) {$u_4$};
\kr{u3}{u4};
 \end{tikzpicture}
&
\begin{tikzpicture}
\jezobc{2};
\wierz{u3}{$(u1)+(1,0)$};
\node [below] at (u3) {$u_3$};
\kr{u1}{u3};
\wierz{u4}{$(u3)+(1,0)$};
\node [below] at (u4) {$u_4$};
\kr{u3}{u4};
\krnwyg{u4}{u0};
 \end{tikzpicture}
&
\makecell[b]{n/a\\ \ }
&
\begin{tikzpicture}
\jezobc{2};
\wierz{u3}{$(u1)+(1,0)$};
\node [below] at (u3) {$u_3$};
\kr{u1}{u3};
\wierz{u4}{$(u3)+(1,0)$};
\node [below] at (u4) {$u_4$};
\kr{u3}{u4};
\krnwyg{u4}{u0};
 \end{tikzpicture}
&
\makecell[b]{
$u_2u_0u_4$\\
\\
\begin{tikzpicture}
\sciez{4};
\krn{a1}{a2};
\node [below] at (a1) {$u_2$};
\node [below] at (a2) {$u_4$};
\node [below] at (a3) {$u_3$};
\node [below] at (a4) {$u_1$};
 \end{tikzpicture}
}
&3\\
\hline
\end{tabular}
} 


\medskip

\paragraph{}\label{nieb1}
First blue edge in round 1.

\medskip\noindent
In this case we have a blue edge $u_0u_1$. In the second round, Builder selects an edge incident to $u_1$ and some 
free vertex $u_2$. If Painter colours  $u_1u_2$ red, we obtain a coloured graph considered in the previous case \ref{nieb2}.
If Painter colours  $u_1u_2$ blue, we have a blue path $u_0,u_1,u_2$, which can be contracted to a blue edge and based
on Lemma \ref{sciag} and the inductive hypothesis, we conclude that Builder wins the game.

\subsection{$G$ is not empty}

It remains to prove the statement if $G$ is a $b$-, $br$-, $brb$-, $brr$-, or $brrb$-path. 
For a $b$-path, the game was analysed in Section \ref{stage1}, case \ref{nieb1}. 
The cases of $br$-, $brb$- and $brrb$-paths were analysed 
in Section \ref{stage1}, case \ref{nieb2}. In the same place we also checked that
Builder wins starting from a $brrr$- or $brrb$-path and hence he wins starting from a $brr$-path as well. 

\section{Algorithm for small instances of $\RR(C_4,P_n,H)$}
\label{algorithm}

We consider a slightly modified version of the game $\RR(C_4,P_n,H)$ introduced in Section \ref{prelim}, which will be denoted by $\RRC(C_4,P_n,H,v,e)$, where $v,e\in\N$.
The board of the game is 
$K_\N$, with exactly $e(H)$ edges coloured, inducing a copy of $H$ in $K_\N$. The rules of selecting and colouring edges
by Builder and Painter are the same as in the standard game $\R(C_4,P_n)$. The additional rules are:
\begin{enumerate}
\item[(1)] 
Builder wins $\RRC(C_4,P_n,H,v,e)$ if after at most $e-|E(H)|$ rounds there is a red $C_4$ or a blue $P_n$ at the board and he loses otherwise.
\item[(2)]
At any moment of the game, the graph induced by coloured edges (including the edges of $H$) has at most $v$ vertices.
\item[(3)] 
At any moment of the game, the graph induced by coloured edges is connected. 
\item[(4)] 
In every round, Builder selects an edge $x$ such that the graph induced by all blue edges and $x$ is contained in a path on $n$ vertices. 
\end{enumerate}

We will also write $rc(C_4,P_n,H,v,e)=1$ if Builder has a winning strategy in $\RRC(C_4,P_n,H,v,e)$ and $rc(C_4,P_n,H,v,e)=0$ otherwise. 
Note that the essential difference between $\RRC(C_4,P_n,H,\infty,2n-2)$ and $\RR(C_4,P_n,H)$ is the connectivity condition. This condition makes the game potentially harder for Builder. Moreover, if $v<v'$, then $rc(G_1,G_2,H,v,e)\le rc(G_1,G_2,H,v',e)$ (extra vertices cannot harm Builder). Our program, described below, computes the values $rc(C_4,P_n,H,n+1,2n-2)$ for $n\in\{7,8,...13\}$ and every $H$ which is one of the coloured graphs: empty graph, $b$-path, $br$-path, $brb$-path, $brr$-path, $brrb$-path.
As the result of the computation, we obtained 
$rc(C_4,P_7,\emptyset,8,12)=0$ and in all other cases $rc(C_4,P_n,H,n+1,2n-2)=1$. It proves Lemma \ref{warpocz}, since the condition $rc(C_4,P_n,H,n+1,2n-2)=1$ implies also a winning strategy of Builder in $\RR(C_4,P_n,H)$.

Using the program, we computed also that $rc(C_4,P_7,\emptyset,8,13)=1$ and it proves the second part of Theorem \ref{main}.

\subsection{Algorithm outline}

The main goal of the algorithm is to find a sufficient Builder's move in every position arising from every possible Painter's strategy. By a position we mean a coloured graph induced by all edges coloured in the game so far. 
We analyse the game tree of $\RRC(C_4,P_n,H,v,e)$ using a recursive method and a simple version of a standard alpha-beta pruning. Here is the outline.

Consider a position $H'$. Create the list of all  Builder choices of his next move, not violating the rules (2), (3) and (4) of the game. Consider the first move, say edge $x$, on the list. Consider every possible colouring of $x$.
For the obtained coloured graph $H'+x$ check if it is isomorphic to a position already considered before.\footnote{In fact, there is no known polynomial algorithm for graph isomorphism problem, so the program looks for isomorphic graphs but doesn't try too hard. The idea is to sort all vertices by blue and red degrees and to look for identical graphs. The visited graphs are stored as a combination of adjacency matrices: the lower-diagonal one represents blue edges and the upper-diagonal one represents red ones. This allows to store a coloured graph in $|V(G)|^2$ bits.}  If so, use the known value $rc(C_4,P_n,H'+x,v,e)$ and prune the further subtree. Otherwise, call the recursive procedure in order to find $rc(C_4,P_n,H'+x,v,e)$. If $rc(C_4,P_n,H'+x,v,e)=1$, prune all other branches of the game tree, starting from the position $H'$; otherwise consider the next edge on the Builder move list in position $H'$ and continue the analysis. 

The algorithm stops if there is a red $C_4$ or a blue $P_n$ at the board, or there are $e-e(H)$ coloured edges. Evaluating a final position $H'$  
is obvious: we put $rc(C_4,P_n,H',v,e)=1$ if $H'$ contains a red $C_4$ or a blue $P_n$; otherwise we put $rc(C_4,P_n,H',v,e)=0$.

\subsection{Implementation}

We implemented the above algorithm in C++ and present the code in Appendix \ref{app:program}.

Let us add that in order to shorten the running time, we create the list of all possible Builder's moves in a position $H$
of the game
$\RRC(C_4,P_{n[t]},H,v[t],e[t])$ in such a way, that if $x$ is a winning move in position $H'$ in a game $\RRC(C_4,P_{n[t-1]},H,v[t-1],e[t-1])$, then $x$ is first at the list while analysing position $H'$ of $\RRC(C_4,P_{n[t]},H,v[t],e[t])$ (functions $n[t],v[t],e[t]$ denote here parameters of $t$-th analysed game in the program). At our standard PC, the program evaluated all games listed in Lemma \ref{warpocz} within about 7 minutes.

\section{Proof of Theorem \ref{main}}

The first part of the main theorem is the consequence of the fact that Builder has a winning strategy in the game $\RR(C_4,P_n)$, as stated in Theorem \ref{main2}. The second part of Theorem \ref{main} follows from the computation $rc(C_4,P_7,\emptyset,8,13)=1$ mentioned in Section \ref{algorithm}, which implies that Builder in $\R(C_4,P_7)$ can force Painter to create a red $C_4$ or a blue $P_7$ within at most 13 rounds.

\bibliographystyle{amsplain}

\appendix
\section{Algorithm implementation}
\label{app:program}

Below we include the C++ implementation of the algorithm presented in section \ref{algorithm}. The program also generates files with winning strategies for Builder in each analysed game (obviously only for games where he has a winning strategy). Each file contains  
some decision trees. Every tree has game parameters as the root  and shows a complete winning strategy for Builder in the root position (of course if he can win).
Each tree is traversed pre-order, but if the program detects that two or more nodes represent isomorphic coloured graphs, then it 
explores the subtree for the first node only.
Though the files with winning Builder strategies are not necessary for proving lemma \ref{warpocz}, they may be useful for creating a perfect AI Builder player. Additionally, these files can be treated as another proof of lemma \ref{warpocz}. However, the size of the files grows exponentially with the number of moves, so it is almost impossible to verify them by a human. That is also the reason why we do not include all the files here. One can find them in a git repository \url{https://github.com/urojony/c4pn}.

Here is an example of such a file. The vertices are labelled as hex numbers $0,1,2,\ldots,C,D$. Every line starting with ``rc'' describes a game and a starting position (a coloured graph $H$). Each line that starts from ``r:'' represents a position in a game that is reachable using the Builder's strategy against any Painter's one. For example, line 5 says that if there is a graph with red edges 0-2, 1-2, and a blue edge 0-1, Builder can win by selecting the edge 0-3. Line 10 says that if there is a graph with red edges 0-1, 0-2, and a blue edge 1-2, Builder can win by choosing 1-3. The last part ``l: 5'' informs that the isomorphic graph was analysed in line 5, so this graph will not be analysed further. 

\lstset{numbers=left,basicstyle=\ttfamily}
\begin{lstlisting}
rc(C4,P3,empty,4,6)=1
r: b: m: 01
 r: b: 01 m: 02
  r: 02 b: 01 m: 21
   r: 02 12 b: 01 m: 03
    r: 02 03 12 b: 01 m: 13
 r: 01 b: m: 02
  r: 01 b: 02 m: 12 l: 4
  r: 01 02 b: m: 12
   r: 01 02 b: 12 m: 13 l: 5
   r: 01 02 12 b: m: 03
    r: 01 02 12 b: 03 m: 23
     r: 01 02 12 23 b: 03 m: 13
    r: 01 02 03 12 b: m: 23
     r: 01 02 03 12 b: 23 m: 13 l: 13
rc(C4,P3,b-path,4,6)=1
 r: b: 01 m: 02 l: 3
rc(C4,P3,br-path,4,6)=1
  r: 12 b: 01 m: 20 l: 4
rc(C4,P3,brr-path,4,6)=1
   r: 12 23 b: 01 m: 13
    r: 12 13 23 b: 01 m: 30 l: 12
rc(C4,P3,brb-path,4,6)=0
\end{lstlisting}

\subsection{Main functions}

The main functions used in our C++ program are: construct, colour.
These are two recurrent functions which simulate moves of Builder and Painter, respectively; they return a true value if the analysed position is winning for Builder and a false value otherwise.

We use also other functions: sortAndMerge, hasOnlyPaths, hasC4, and an object ANAL\_POS which is a vector of maps. Here are their descriptions:

\begin{itemize}
    \item 
    sortAndMerge(blueG,redG) sorts the vertices in the graph by blue degrees and then by red degrees, then compresses two adjacency matrices into one matrix;
    \item
     hasOnlyPaths(blueG,edge) that checks if that graph blueG+edge is a sum of vertex-disjoint paths provided blueG is such a graph;
     \item
     hasC4(redG,edge) that checks if the graph redG+edge contains $C_4$ provided redG does not;
     \item
     ANAL\_POS[$T$] is a map of analysed positions and winning moves for Builder in the game number $T$; ANAL\_POS[T][position] is 0, if in the game number $T$ in the given position there is no winning  move for Builder. 
\end{itemize}




\subsection{C++ program}


{\small
\lstset{language=C++,
                basicstyle=\ttfamily,
                keywordstyle=\color{blue}\ttfamily,
                stringstyle=\color{red}\ttfamily,
                commentstyle=\color{green}\ttfamily,
                morecomment=[l][\color{magenta}]{\#}
}
\begin{lstlisting}
#include<vector>
#include<array>
#include<bitset>
#include<set>
#include<unordered_set>
#include<unordered_map>
#include<algorithm>
#include<cstdio>
#include<string>
#include<map>
using namespace std;
///Tmax - number of considered games
///T2max - number of considered starting positions in each game
constexpr int T_MAX=12,T2_MAX=6;
///Parameters of considered games - the T-th game
///is played on V[T] vertices, the goal for the builder is to build
///a red C4 or a blue path of length N[T] in E[T] moves.
///N should be sorted in nondescending order for faster performance
///For any T, V[T]>N[T] (otherwise Painter wins trivially or easy)
constexpr int V[]={4,5,6,7,8,8,9,10,11,12,13,14},
E[]={6,8,9,11,13,12,14,16,18,20,22,24},N[]={3,4,5,6,7,7,8,9,10,11,12,13};
///Vmax has to be the biggest number in V[0],V[1],...,V[Vmax-1]
///If it's not V[Tmax-1], change manually
constexpr int V_MAX=V[T_MAX-1],V_MAX_SQ=V_MAX*V_MAX;
int T=0;
///will be treated as Vmax*Vmax matrix (usually representing a graph)
typedef bitset<V_MAX_SQ> matrix;
///list of starting positions: j-th position is stored in RED_EDGES[j]
///and BLUE_EDGES[j], each edge is stored as two consecutive numbers
///so e.g. {0,1,2,3} means two edges: 0-1 and 2-3;
///if a position is not possible in a given game, it will be ignored
///for example if V[0]=4 (so the set of vertices={0,1,2,3}) then all
///positions in the 0-th game which contains a vertex 4 will be ignored
const vector<int> RED_EDGES[] {{},{},{1,2},{1,2,2,3},{1,2},{1,2,2,3}},
BLUE_EDGES[] {{},{0,1},{0,1},{0,1},{0,1,2,3},{0,1,3,4}};
#define ST first
#define ND second

///map of all analysed positions with builder to move
///key=position
///value=0 if builder has no winning move or (v1<<8)+v2 if (v1,v2) is winning
unordered_map<matrix,int> ANAL_POS[T_MAX];
matrix strictlyLowerTriangularMask() {
   matrix low;
   for(int i=1;i<V_MAX;++i)
      for(int j=0;j<i;++j)
         low[i*V_MAX+j]=1;
   return low;
}
const matrix LOWER_MASK=strictlyLowerTriangularMask();
matrix strictlyUpperTriangularMask(){
   matrix upp;
   for(int i=0;i<V_MAX;++i)
      for(int j=i+1;j<V_MAX;++j)
         upp[i*V_MAX+j]=1;
   return upp;
}
const matrix UPPER_MASK=strictlyUpperTriangularMask();
const array<matrix,V_MAX> verticesMasks(){
///masks for each row of a matrix
   array<matrix,V_MAX> ver;
   int i=0;
   for(i=0;i<V_MAX;++i)
      ver[0][i]=1;
   for(i=1;i<V_MAX;++i)
      ver[i]=ver[i-1]<<V_MAX;
   return ver;
}
const array<matrix,V_MAX> ROW_MASKS=verticesMasks();
int degree(matrix& graph,int v){
   return (graph&ROW_MASKS[v]).count();
}
struct Graph{
   matrix m; ///adjacency matrix
   int e=0; ///number of edges
   void addEdge(int v1,int v2){
      m[v1*V_MAX+v2]=m[v2*V_MAX+v1]=1;
      ++e;
   }
   void removeEdge(int v1,int v2){
      m[v1*V_MAX+v2]=m[v2*V_MAX+v1]=0;
      --e;
   }
   void clear(){
      m.reset();
      e=0;
   }
   int degree(int v){
      return ::degree(m,v);
   }
};

void reorder(matrix &graph,int invOrder[]){
///reorders vertices in a graph - useful for decreasing
///the number of unique positions in the game
   matrix res;
   for(int i=graph._Find_first();i<V_MAX_SQ;i=graph._Find_next(i))
      res[invOrder[i/V_MAX]*V_MAX+invOrder[i%V_MAX]]=1;
   graph=res;
}
matrix sortAndMerge
(matrix blueG,matrix redG,int v,int invOrder[],pair<int,int> order[]){
///sorts vertices of matrices by blue degree, then by red degree
///after that it merges two matrices into one
///it sends back the merged matrix and the order of sorting
   int i;
   for(i=0;i<v;++i)
      order[i]={-degree(blueG,i)*4194304-degree(redG,i)*65536,i};
   sort(order,order+v);
   for(i=0;i<v;++i)
      invOrder[order[i].ND]=i;
   reorder(blueG,invOrder);
   reorder(redG,invOrder);
   return (blueG&LOWER_MASK) | (redG&UPPER_MASK);
}
bool isPnPath(matrix graph){
///if there is an exactly one path of length Nt[T] and
///no other edges, returns 1,
///if there is no path of length Nt[T], returns 0,
///otherwise it can return any value
   //return 1; ///if Vt[T]==Nt[T]+1, it should always return 1 when invoked
   int i,j,k,len=0;
   for(i=0;i<V[T];++i)
      if(degree(graph,i)==1)
         break;
   if(i==V[T])
      return 0;
   for(j=i;;){
      k=graph._Find_next(j*V_MAX-1)-j*V_MAX;
      if(k>=V_MAX)
         break;
      graph[j*V_MAX+k]=graph[k*V_MAX+j]=0;
      ++len;
      j=k;
   }
   return len+1==N[T];
}
bool hasC4(Graph& graph,int i,int j){
///checks if graph has a C4 cycle which contains i-j edge
   if(graph.e<3)
      return 0;
   auto graph2=graph.m&ROW_MASKS[i];
   ///macro for iterating a bitset over interval [beg,end)
   #define FOR_BS(i,bset,beg,end)\
   for(int i=(bset)._Find_next((beg)-1);i<end;i=(bset)._Find_next(i))
   FOR_BS(k,graph.m,j*V_MAX,j*V_MAX+i){
      if(((graph2>>((i-k+j*V_MAX)*V_MAX))&graph.m).count()>1)
         return 1;
   }
   FOR_BS(k,graph.m,j*V_MAX+i+1,j*V_MAX+V_MAX){
      if(((graph2<<((k-j*V_MAX-i)*V_MAX))&graph.m).count()>1)
         return 1;
   }
   return 0;
}
bool hasOnlyPaths(matrix graph,int v1,int v2){
///assuming graph contains only disjoint paths, check if
///the edge v1-v2 doesn't destroy this property
   int j,k;
   const int d1=degree(graph,v1),d2=degree(graph,v2);
   if(d1>1 || d2>1)
      return 0;
   if(d1==0 || d2==0){
      return 1;
   }
   for(j=v1;;){
      k=graph._Find_next(j*V_MAX-1)-j*V_MAX;
      if(k>=V_MAX)
         break;
      graph[j*V_MAX+k]=graph[k*V_MAX+j]=0;
      j=k;
   }
   return j!=v2;
}
int NUM_OF_UNIQ_POS=0,NUM_OF_ALL_POS=0;
bool colour(Graph blueG,Graph redG,int v,int i,int j);
bool construct(Graph blueG,Graph redG,int v){
///returns 1 iff builder has a winning move
   if(blueG.e>=N[T] || redG.e>E[T]-N[T]+1)
      return 0;
   int invOrder[v+3];
   pair<int,int> order[v+3];
   auto sPos=sortAndMerge(blueG.m,redG.m,v,invOrder,order);
   ++NUM_OF_ALL_POS;
   if(ANAL_POS[T].count(sPos))
      return ANAL_POS[T][sPos];
   ++NUM_OF_UNIQ_POS;
   if((NUM_OF_UNIQ_POS&((1<<23)-1))==0)
      printf("Unique positions analysed so far: %d\n",NUM_OF_UNIQ_POS);
   if(blueG.e+redG.e==E[T])
      return 0;
   order[v].ND=v;
   order[v+1].ND=v+1;
   matrix checked=blueG.m|redG.m;
   if(T>0 && ANAL_POS[T-1].count(sPos) && ANAL_POS[T-1][sPos]){
   ///checks if in the previous game this position was reached
      int moove=ANAL_POS[T-1][sPos];
      int v1=order[moove>>8].ND,v2=order[moove&255].ND;
      if(colour(blueG,redG,max({v,v1+1,v2+1}),v1,v2)){
      ///try to play the same move
         ANAL_POS[T][sPos]=moove;
         return 1;
      }
      else
         checked[v1*V_MAX+v2]=checked[v2*V_MAX+v1]=1;
   }
   ///try moves with vertices that already have edges
   for(int i=v-1;i>=0;--i)
      if(blueG.degree(i)<=1)
         for(int j=v-1;j>i;--j)
            if(!checked[i*V_MAX+j] && blueG.degree(j)<=1){
               if(colour(blueG,redG,v,i,j)){
                  ANAL_POS[T][sPos]=(invOrder[i]<<8)+invOrder[j];
                  return 1;
               }
               else
                  checked[i*V_MAX+j]=checked[j*V_MAX+i]=1;
            }
   ///try moves with an unused vertice
   if(v<V[T])
      for(int i=0;i<v;++i){
         if(blueG.degree(i)<=1&& !checked[i*V_MAX+v]
         && colour(blueG,redG,v+1,i,v)){
            ANAL_POS[T][sPos]=(invOrder[i]<<8)+v;
            return 1;
         }
      }
   ///the first move has two unused vertices
   if(v==0 && colour(blueG,redG,v+2,v,v+1)){
      ANAL_POS[T][sPos]=(v<<8)+v+1;
      return 1;
   }
   ANAL_POS[T][sPos]=0;
   return 0;
}
void print(const matrix& mergedG,int moove){
   printf("r: ");
   int i,j;
   for(i=0;i<V_MAX;++i)
      for(j=i+1;j<V_MAX;++j)
         if(mergedG[i*V_MAX+j])
            printf("%X%X ",i,j);
   printf("b: ");
   for(i=1;i<V_MAX;++i)
      for(j=0;j<i;++j)
         if(mergedG[i*V_MAX+j])
            printf("%X%X ",i,j);
   printf("m: %X%X\n",moove>>8,moove&255);
}
void print(const matrix& blueG,const matrix& redG,
            int moove,FILE *fp,bool newline=1){
   fprintf(fp,"r: ");
   int i,j;
   for(i=0;i<V_MAX;++i)
      for(j=i+1;j<V_MAX;++j)
         if(redG[i*V_MAX+j])
            fprintf(fp,"%X%X ",i,j);
   fprintf(fp,"b: ");
   for(i=0;i<V_MAX;++i)
      for(j=i+1;j<V_MAX;++j)
         if(blueG[i*V_MAX+j])
            fprintf(fp,"%X%X ",i,j);
   fprintf(fp,"m: %X%X",moove>>8,moove&255);
   if(newline)
      fprintf(fp,"\n");
}
bool colour(Graph blueG,Graph redG,int v,int i,int j){
///returns 0 iff painter has a winning move
   if(!hasOnlyPaths(blueG.m,i,j))
      return 0;
   blueG.addEdge(i,j);
   bool isVertexWithoutBlueEdges=0;
   for(int i=V[T]-1;i>=0;--i)
      if(blueG.degree(i)==0){
         isVertexWithoutBlueEdges=1;
         break;
      }
   if(!isVertexWithoutBlueEdges) ///if all vertices have blue edges,
      return 0; ///then at least one edge will be wasted
   ///paint i-j red and check who's winning
   blueG.removeEdge(i,j);
   redG.addEdge(i,j);
   if(!hasC4(redG,i,j) && !construct(blueG,redG,v)){
      return 0;
   }
   ///paint i-j blue and check who's winning
   redG.removeEdge(i,j);
   blueG.addEdge(i,j);
   if((blueG.e+2<N[T] || !isPnPath(blueG.m))&& !construct(blueG,redG,v)){
      return 0;
   }
   return 1;
}
unordered_map<matrix,int> BOOK_POS;
int LINE_NUMBER;
void printBook(Graph blueG,Graph redG,int v,int e,FILE *fp){
   int invOrder[v+3];
   pair<int,int> order[v+3];
   auto sPos=sortAndMerge(blueG.m,redG.m,v,invOrder,order);
   order[v].ND=v;
   order[v+1].ND=v+1;
   if(!ANAL_POS[T].count(sPos)){
      construct(blueG,redG,v);
   }
   int i,j,moove;
   moove=ANAL_POS[T][sPos];
   if(moove==0){
      printf("Can't find winning strategy for Builder\n");
      exit(2137); ///this shouldn't happen
   }
   i=moove/256;
   j=moove%256;
   i=order[i].ND;
   j=order[j].ND;
   for(int h=0;h<e;++h)
      fprintf(fp," ");
   print(blueG.m,redG.m,i*256+j,fp,0);
   if(BOOK_POS.count(sPos)){
      fprintf(fp," l: %d\n",BOOK_POS[sPos]);
      ++LINE_NUMBER;
      return;
   }
   else{
      fprintf(fp,"\n");
      BOOK_POS[sPos]=++LINE_NUMBER;
   }
   blueG.addEdge(i,j);
   if(!isPnPath(blueG.m))
      printBook(blueG,redG,max({v,i+1,j+1}),e+1,fp);
   blueG.removeEdge(i,j);
   redG.addEdge(i,j);
   if(!hasC4(redG,i,j))
      printBook(blueG,redG,max({v,i+1,j+1}),e+1,fp);
}
int fillGraph(Graph& blueG,Graph& redG,const int t2){
   blueG.clear();
   redG.clear();
   for(int i=1;i<(int)BLUE_EDGES[t2].size();i+=2)
      blueG.addEdge(BLUE_EDGES[t2][i],BLUE_EDGES[t2][i-1]);
   for(int i=1;i<(int)RED_EDGES[t2].size();i+=2)
      redG.addEdge(RED_EDGES[t2][i],RED_EDGES[t2][i-1]);
   int mxVertexIndex=-1;
   for(int i:BLUE_EDGES[t2])
      mxVertexIndex=max(i,mxVertexIndex);
   for(int i:RED_EDGES[t2])
      mxVertexIndex=max(i,mxVertexIndex);
   return mxVertexIndex;
}
int main(){
   Graph blueG,redG;
   FILE *fp;
   for(T=0;T<T_MAX;++T){
      BOOK_POS.clear();
      LINE_NUMBER=0;
      char filename[100];
      char names[6][20]={"empty","b-path","br-path",
      "brr-path","brb-path","brrb-path"};
      sprintf(filename,"C4P%d_in_%d_moves_%d_verts.txt",N[T],E[T],V[T]);
      fp=fopen(filename,"w");
      for(int t2=0;t2<T2_MAX;++t2){
         int mxVertexIndex=fillGraph(blueG,redG,t2);
         if(mxVertexIndex>=V[T])
            continue;
         int result=construct(blueG,redG,mxVertexIndex+1);
         fprintf(fp,"rc(C4,P%d,%s,%d,%d)=%d\n",N[T],names[t2],V[T],E[T],result);
         printf("rc(C4,P%d,%s,%d,%d)=%d\n",N[T],names[t2],V[T],E[T],result);
         printf("unique/total positions analysed: ");
         printf("%d %d\n",NUM_OF_UNIQ_POS,NUM_OF_ALL_POS);
         ++LINE_NUMBER;
         if(result)
            printBook(blueG,redG,mxVertexIndex+1,
               BLUE_EDGES[t2].size()/2+RED_EDGES[t2].size()/2,fp);
      }
      fclose(fp);
   }
}

\end{lstlisting}
}
\end{document}